\documentclass[a4paper,12pt]{amsart}
\usepackage{amsmath}
\usepackage{amsthm}
\usepackage{amssymb}
\usepackage{amscd}

\newtheorem{thm}{Theorem}[section]

\newtheorem{prop}[thm]{Proposition}
\newtheorem{lemma}[thm]{Lemma}
\newtheorem{cor}[thm]{Corollary}

\newtheorem{rmk}[thm]{Remark}

\newtheorem{Problem}{Problem}[section]
\newcommand{\proj}{\mathop{\rm Proj}\nolimits}
\newcommand{\im}{\mathop{\rm Im}\nolimits}
\newcommand{\Pic}{\mathop{\rm Pic}\nolimits}

\DeclareMathOperator{\Hom}{Hom}
\DeclareMathOperator{\Ext}{Ext}
\DeclareMathOperator{\RHom}{RHom}
\DeclareMathOperator{\lotimes}{\otimes^{\mathbb{L}}}
\newcommand{\caltor}{\mathop{{\mathcal T\!or}}\nolimits}
\DeclareMathOperator{\Ker}{Ker}
\DeclareMathOperator{\Coker}{Coker}
\DeclareMathOperator{\End}{End}

\DeclareMathOperator{\rk}{rank}

\DeclareMathOperator{\Gr}{Gr}
\DeclareMathOperator{\HN}{HN}
\DeclareMathOperator{\ch}{char}

\title[Nef vector bundles on a projective space or a hyperquadric]
{Nef vector bundles on a projective space or a hyperquadric
with the first Chern class small}

\thanks{
This work was partially supported by 
JSPS KAKENHI (C) Grant Numbers 22540043
and 15K04810.
}

\author{Masahiro Ohno
}

\address{Graduate School of Informatics and Engineering,
The University of Electro-Communications,
Chofu-shi,
Tokyo, 182-8585 Japan
}

\email{masahiro-ohno@uec.ac.jp}



\begin{document}
\begin{abstract}
We give a new proof of the classification 
due to  Peternell-Szurek-Wi\'{s}niewski
of nef vector bundles on a projective space
with the first Chern class less than three
and on a smooth hyperquadric
with the first Chern class less than two
over an algebraically closed field of characteristic zero.
\end{abstract}

\maketitle

\section{Introduction}
Let $X$ be a smooth complex projective variety of dimension $n$,
and 
$\mathcal{E}$ 
a nef vector bundle of rank $r$ on it.
Many authors have studied a pair $(X,\mathcal{E})$
of $X$ and $\mathcal{E}$
in connection with classifications of special types of Fano (or Fano-like) manifolds,
where $\mathcal{E}$
may be assumed
to be 
spanned (i.e., globally generated)
or ample,
as the case may be.
In these classifications, often appears a pair $(X, \mathbb{E})$ of 
$X$ of Picard number one
and an ample vector bundle
$\mathbb{E}$ with the adjoint bundle $K_X+\det\mathbb{E}$ trivial (i.e., isomorphic to the structure sheaf).
In case 
$K_X+\det\mathbb{E}$ is trivial, 
we have $r\leqq n+1$ by Mori's theory of extremal rays~\cite{mo2},
and such pairs $(X,\mathbb{E})$ are classified in the cases $r=n+1$, $n$, and $n-1$
by Ye-Zhang~\cite{yz}, Peternell~\cite{p0} \cite{p}, Fujita~\cite{fn}, and Peternell-Szurek-Wi\'{s}niewski~\cite{psw}.
Note here that 
the 
projective space bundle $\mathbb{P}(\mathbb{E})$ is a Fano manifold;
in such a case, $\mathbb{E}$ is called a Fano bundle.
In this vein, in view of  $-K_{\mathbb{P}(\mathbb{E})}=rH(\mathbb{E})$,
$\mathbb{E}$ might be called a ``Del-Pezzo bundle''
if $r=n-1$, and a ``Mukai bundle''
if $r=n-2$.
As is well known, 
Mukai manifolds,
i.e., Fano manifolds of coindex three,
are 
described in
\cite{MR0995400},
whereas 
the ``Mukai bundle'' $\mathbb{E}$
has not been investigated for an arbitrary rank $r$
even if the underlying manifold
$X$ is simple such as $\mathbb{P}^n$ or $\mathbb{Q}^n$.
Thus the present
deepest result in this direction 
is the classification of ``Del-Pezzo bundles'' due to  Peternell-Szurek-Wi\'{s}niewski~\cite{psw}.
Roughly
speaking,
their method of classification is to relate
the pair $(X,\mathbb{E})$
with a pair $(X,\mathcal{E})$ by setting $\mathcal{E}=\mathbb{E}(-1)$,
and show that  $\mathcal{E}$ is nef by their Comparison lemma~\cite[(3.1)]{psw}:
an ample $\mathbb{E}$ can be replaced by a nef $\mathcal{E}$.
Then they 
classified 
such $(X,\mathcal{E})$'s
in \cite{pswnef}
in the following cases:
if $X$ is isomorphic to a projective  $\mathbb{P}^n$ and 
the first Chern class $c_1(\mathcal{E})$ of $\mathcal{E}$ is less than three;
if $X$ is isomorphic to a hyperquadric $\mathbb{Q}^n$ ($n\geqq 3$) and 
$c_1(\mathcal{E})< 2$.
(Here, for simplicity as in \cite{pswnef}, we identify the first Chern class $c_1(\mathcal{E})$
with the corresponding non-negative integer via the isomorphism $\Pic X\cong \mathbb{Z}$
by abuse of notation.)
Thus the classification~\cite{pswnef} of the pairs $(X,\mathcal{E})$ in the above cases has fundamental importance in their proof.
Towards the classification of ``Mukai bundles'',
it seems therefore natural to consider the classification of $(X,\mathcal{E})$ in the next cases,
e.g., $X=\mathbb{P}^n$ and $c_1(\mathcal{E})=3$,
or $X=\mathbb{Q}^n$ and $c_1(\mathcal{E})=2$.
However the classification  of $(X,\mathcal{E})$ with $\mathcal{E}$ nef
of an arbitrary rank $r$
in the next cases has not been pursued over twenty years.

One of the reason why such research has not been pursued seems to come from the fact that 
it is uncertain how to describe nef bundles in general.
In order to 
overcome 
this situation, in this paper, we first 
propose a framework 
to describe nef bundles 
on a projective space or a smooth hyperquadric.
Following this framework, we secondly
give a new proof of the above classification of $(X,\mathcal{E})$ 
due to Peternell-Szurek-Wi\'{s}niewski~\cite{pswnef}.
In their proof, 
Peternell-Szurek-Wi\'{s}niewski 
analyze the contraction morphism of an extremal ray
of the Fano manifold $\mathbb{P}(\mathcal{E})$.
On the other hand, our proof depends only on the cohomological study of $\mathcal{E}$ with respect to 
a full strong exceptional sequence of vector bundles, and does not analyze any contraction morphism.
(See Theorems~\ref{d=1onProSpace}, \ref{d2OnProSpace}, and \ref{d=1OnHyperquadric}
and the proofs therein.)

More precisely our framework is based on the following observation.
Suppose there exists a full strong exceptional sequence
$G_0,\dots,G_m$ 
of vector bundles on $X$.
Recall here that 
a projective space or a smooth hyperquadric
admits such a full strong exceptional sequences 
of vector bundles
by \cite{MR0509388} and \cite{MR0939472}.
Denote by $G$ the direct sum $\bigoplus _{i=0}^m G_i$,
and by $A$ the endomorphism ring $\End(G)$ of $G$.
Let $F$ be a coherent sheaf on $X$. 
Then Bondal's theorem \cite[Theorem 6.2]{MR992977} implies an isomorhism
\[\RHom(G,F)\lotimes_A G\cong F.\]
Suppose first that $\Ext^q(G,F)=0$ for all $q>0$.
Then $\RHom(G,F)\cong \Hom(G,F)$, and 
the isomorphism above implies that a projective resolution of the right $A$-module $\Hom(G,F)$
induces the following locally free resolution of $F$:
\[0\to G_0^{\oplus e_{m,0}}\to\dots\to \bigoplus_{j=0}^{m-l}G_{j}^{\oplus e_{l,j}}\to\dots\to
\bigoplus_{j=0}^{m}G_{j}^{\oplus e_{0,j}}\to F\to 0\]
where $e_{0,j}=\dim \Hom(G_j,F)$ for all $j=0,\dots,m$
and, for any $l\geqq 1$ and any $j\leqq m-l$, $e_{l,j}$ is determined inductively by the following formula:
$e_{l,j}=\sum_{j<k}e_{l-1,k}\dim \Hom(G_j,G_k)$.
For an arbitrary coherent sheaf $F$, 
we have $\Ext^q(G,F(d))=0$ for all $q>0$ if $d$ is sufficiently large by Serre's vanishing.
Therefore we have a resolution of $F$ of the above form
by replacing $(G_0,\dots, G_m)$ by a new full strong exceptional sequence $(G_0(-d),\dots,G_m(-d))$.
Note here that Serre's vanishing does not give, in general, an effective estimate
of the integer $d$. 
However, on a projective space
or a smooth hyperquadric, 
the full strong exceptional sequence $G_0,\dots,G_m$
can be chosen to be that of well-understood vector bundles 
so that, by applying the Kodaira or Kawamata--Viehweg vanishing theorem, we can give an 
\textit{effective} 
estimate
of the integers $d$
such that $\Ext^q(G,\mathcal{E}(d))=0$ for all $q>0$ 
for a \textit{nef} vector bundle $\mathcal{E}$
(Corollaries~\ref{ProSpaceResol}, \ref{quadricResol3}, and \ref{quadricResol}).

Let us look at our proof more closely, e.g., in the case where 
$X$ is an odd dimensional smooth complex hyperquadric 
$\mathbb{Q}^n$;
in this case, we can take the sequence $(G_0,\dots,G_m)$ 
to be
$(\mathcal{O},\mathcal{S},\mathcal{O}(1), \dots,\mathcal{O}(n-1))$
where $\mathcal{S}$ is the (spanned) spinor bundle.
For a nef vector bundle $\mathcal{E}$ on $X$, 
let $d_{\min}$ be the minimal integer $d_{\min}$ such that $\Ext^q(G,\mathcal{E}(d'))=0$ for all $q>0$ and all $d'\geqq d_{\min}$.
Then the Kodaira vanishing theorem shows that $d_{\min}\leqq c_1(\mathcal{E})$
(Corollary~\ref{quadricResol}),
and we have the following 
locally free resolution
\[0\to G_0^{\oplus e_{m,0}}\to\dots\to \bigoplus_{j=0}^{m-l}G_{j}^{\oplus e_{l,j}}\to\dots\to
\bigoplus_{j=0}^{m}G_{j}^{\oplus e_{0,j}}\to \mathcal{E}(d_{\min})
\to 0.\]
By tensoring $\mathcal{O}(-d_{\min})$,
we get a locally free resolution of $\mathcal{E}$
in term of a full strong exceptional sequence 
$(\mathcal{O}(-d_{\min}),\mathcal{S}(-d_{\min}),\mathcal{O}(1-d_{\min}), \dots,\mathcal{O}(n-d_{\min}))$.
Moreover the fact that $\mathcal{E}$ is nef imposes several constraints on $e_{l,j}$'s and $d_{\min}$;
some easy constraints are that $e_{l,j}=0$ if 
$l+j>d_{\min}+c_1(\mathcal{E})+1$ (Propositions~\ref{easyConst} and \ref{firstConstraint} (2) (a))
and that $d_{\min}\geqq 0$ if 
$c_1(\mathcal{E})<r$ (Proposition~\ref{firstConstraint} (2) (d)).
Therefore if 
$c_1(\mathcal{E})$
is small then 
$d_{\min}$ has very few possible values
 and most $e_{l,j}$'s vanish.
Note that the above resolution 
contains superfluous direct summands,
so that we have to remove redundant direct summands.
If 
$c_1(\mathcal{E})=1$, other constraints among $e_{l,j}$'s and a more detailed analysis of the resolution
enable us to do so and we get the desired 
resolution as described in the classification due to Peternell-Szurek-Wi\'{s}niewski~\cite{pswnef}.

In
the subsequent papers \cite{Nefofc1=3OnPN} and \cite{Nefofc1=(21)OnQ2},
following our framework,
we classify nef vector bundles on a projective space with the first Chern class three
and the second Chern class less than eight,
and nef vector bundles on a smooth quadric surface with the first Chern class $(2,1)$.

\subsection{Notation and conventions}\label{convention}
In this paper, 
we work over an algebraically closed field $K$.
Basically we follow the standard notation and terminology in algebraic
geometry. 
For example, for a vector bundle $\mathcal{E}$,
$\mathbb{P}(\mathcal{E})$ denotes $\proj S(\mathcal{E})$,
where $S(\mathcal{E})$ denotes the symmetric algebra of $\mathcal{E}$.
The tautological line bundle $\mathcal{O}_{\mathbb{P}(\mathcal{E})}(1)$
is also denoted by $H(\mathcal{E})$.
For a 
coherent sheaf $\mathcal{F}$ on a smooth projective variety $X$,
we denote by $c_i(\mathcal{F})$ the $i$-th Chern class of $\mathcal{F}$.
For a smooth projective variety $X$,
denote by $D^b(X)$ the bounded derived category of the abelian category of coherent sheaves on $X$,
and call $D^b(X)$ the bounded derived category of coherent sheaves on $X$ for short.
We say that a vector bundle is spanned or globally generated
if it is generated by global sections.
For ``spinor bundles'',
we  follow Kapranov's convention~\cite{MR0939472};
our spinor bundles are spanned
and they are the 
duals
of those of Ottaviani's~\cite{ot}.
See \cite[\S 5 Definition 1]{MR3275418} for a precise definition of our spinor bundles.
Finally we refer to \cite{MR2095472} for the definition
and basic properties of nef vector bundles.

%
%

\section{Preliminaries}\label{Preliminaries}
Let $X$ be an $n$-dimensional smooth projective variety over $K$,
and suppose that there exists a full strong exceptional sequence
$G_0,\dots,G_m$ of vector bundles on $X$.

Recall here the definition of a full strong exceptional sequence.
An object $G_i$ of the bounded derived category $D^b(X)$ 
of coherent sheaves on $X$ is said to be 
exceptional if $\RHom (G_i, G_i)\cong K$
and a sequence $G_0,\dots,G_m$ of exceptional objects 
is said to be exceptional
if $\RHom (G_i, G_j)=0$ for all $0\leqq j<i\leqq m$.
An exceptional sequence $G_0,\dots,G_m$ is said to be strong
if $\Ext^k(G_i, G_j)=0$ for all $k>0$ and $0\leqq i<j\leqq m$.
Finally a strong exceptional sequence $G_0,\dots,G_m$
is said to be full if $D^b(X)$ is the smallest triangulated full subcategory
containing $G_0,\dots,G_m$ and closed under isomorphism. 

If $X$ is an $n$-dimensional projective space $\mathbb{P}^n$, 
then $(\mathcal{O},\mathcal{O}(1),\dots,\mathcal{O}(n))$
is a strong exceptional sequence of line bundles,
and it is full by Beilinson's theorem~\cite[Theorem]{MR0509388}.
If $X$ is an odd-dimensional smooth hyperquadric $\mathbb{Q}^n$
and the characteristic $\ch K$ of the base field $K$ is zero,
then it follows from Bott's vanishing theorem 
that $(\mathcal{O},\mathcal{S},\mathcal{O}(1),\dots,\mathcal{O}(n-1))$
is a strong exceptional sequence of vector bundles,
where $\mathcal{S}$ is the spinor bundle,
and it is full by Kapranov's theorem~\cite[Theorem 4.10]{MR0939472}.
If $X$ is an even-dimensional smooth hyperquadric $\mathbb{Q}^n$
and $\ch K=0$,
then 
Bott's vanishing theorem shows
that $(\mathcal{O},\mathcal{S}^+,\mathcal{S}^-,\mathcal{O}(1),\dots,\mathcal{O}(n-1))$
is a strong exceptional sequence of vector bundles,
where $\mathcal{S}^+$ and $\mathcal{S}^-$ are spinor bundles,
and it is full by Kapranov's theorem~\cite[Theorem 4.10]{MR0939472}.
Recall here that we follow Kapranov's convention for ``spinor bundles'';
for example,
on a smooth quadric surface $\mathbb{Q}^2$, spanned line bundles $\mathcal{O}(1,0)$
and $\mathcal{O}(0,1)$ are spinor bundles,
and thus $(\mathcal{O},\mathcal{O}(1,0),\mathcal{O}(0,1),\mathcal{O}(1))$
is a full strong exceptional sequence on $\mathbb{Q}^2$.

For some other fundamental facts about derived categories,
we refer to an excellent book 
\cite{MR2182076} 
of Kashiwara-Schapira
as a literature written in English.

Denote by $G$ the direct sum $\bigoplus_{i=0}^mG_i$ of $G_0,\dots,G_m$,
and by $A$ the endomorphism ring 
$\End(G)$ of $G$.
Then $A$ is a finite-dimensional $K$-algebra.
We refer to \cite[Chap. I, II, III]{MR2197389} for some 
basic facts about modules over a finite-dimensional $K$-algebra.
Note 
that we follow the convention that 
the composite of two arrows $\alpha:a\to b$ and $\beta:b\to c$
is denoted by $\beta\alpha$.
In the same vein, we regard $G$ as a left $A$-module.

For a coherent sheaf $F$ on $X$, 
Bondal's theorem \cite[Theorem 6.2]{MR992977} implies that 
\[
\RHom(G, F)\lotimes_A F\cong F,
\]
so that if $\Ext^q(G,F)=0$ for all $q>0$
then $\RHom(G, F)\cong \Hom(G, F)$,
and a projective resolution of the right $A$-module $\Hom(G, F)$
will play a key role in this paper;
let us recall here briefly a projective resolution of a right $A$-module.

Let $p_i:G\to G_i$ be the projection,
and $\iota_i:G_i\hookrightarrow G$ the inclusion.
Set $e_i=\iota_i\circ p_i$ in $A$.
Denote $e_iA$ by $P_i$.
Then $P_i\cong \Hom(G,G_i)$ as right $A$-modules,
and $A=\bigoplus_{i=0}^m P_i$;
$P_i$ is projective and $P_i\otimes_AG\cong G_i$.

For a finitely generated right $A$-module $V$,
a right $A$-submodule $V^{\leqq i}$ of $V$ is defined 
by the formula $V^{\leqq i}=\bigoplus_{j\leqq i}Ve_j$.
We have a natural isomorphism $V^{\leqq i}\cong V\otimes_A A^{\leqq i}$,
and associated to every module $V$ is 
an ascending filtration 
\[0=V^{\leqq -1}\subset V^{\leqq 0}\subset V^{\leqq 1}\subset \dots\subset V^{\leqq m}=V\]
by right $A$-submodules. 
Set $\Gr^iV=V^{\leqq i}/V^{\leqq i-1}$; $\Gr^iV$ is a right $A$-module.
Denote by $V^i$ the $K$-vector subspace $Ve_i$ of $V$.
Note that $V^i$ is not a $A$-submodule of $V$,
but we have an isomorphism $\Gr^iV\cong V^i$ of $K$-vector spaces.
For example, we have 
\[P_k^{\leqq i}\cong \bigoplus_{j\leqq i}\Hom(G_j, G_k)
\textrm{ and }P_k^j\cong \Hom(G_j, G_k).\]
Note in particular that $P_k^k\cong K$ and that $P_k^j=0$ if $j>k$.
For a homomorphism $\varphi:V\to W$ of right $A$-modules,
we denote by $\varphi^i$ the induced homomorphism $V^i\to W^i$ of $K$-vector spaces.

We have a natural right $A$-linear map
\[\varphi_{i,V}:V^i\otimes_KP_i\to V\]
sending $v\otimes a$ to $va$.
Since every element $v$ of $V^i$ can be written as $v=v'e_i$ for some $v'\in V$, 
we see $\varphi_{i,V}(v\otimes e_i)=v$.
Hence the induced $K$-linear $\varphi_{i,V}^i$ is an isomorphism:
\[(V^i\otimes_KP_i)^i\cong V^i\otimes_KP_i^i\cong V^i.\]
All $\varphi_{i,V}$ together give a canonical surjection
\[\varphi_V:\bigoplus_{j}V^j\otimes_KP_j\to V.\]
Set $W=\Ker \varphi_V$, and consider the canonical surjection
\[\varphi_W:\bigoplus_{i}W^i\otimes_KP_i\to W\]
for $W$.
Here, for a non-zero $V$, define $d(V)$ by the following formula:
\[d(V)=\max \{j\in \mathbb{Z}_{\geqq 0}|V^j\neq 0\}.\]
Since $\varphi_V^i$ is a surjective $K$-linear
$\oplus_{i\leqq j}V^j\otimes_KP_j^i\cong (\oplus_{j\leqq m}V^j\otimes_KP_j)^i\to V^i$,
we see that 
\[\dim W^i
=\sum_{i<j\leqq d(V)}\dim V^j\dim P_j^i
=\sum_{i<j\leqq d(V)}\dim V^j\dim \Hom(G_i,G_j).
\]
In particular, $W^{i}=0$ for all $i\geqq d(V)$.
These consideration leads to the following.
\begin{lemma}\label{projresol}
Every finitely generated right $A$-module $V$
has a bounded projective resolution of the following form
\[0\to P_0^{\oplus e_{m,0}}\to\dots\to \bigoplus_{j=0}^{m-l}P_{j}^{\oplus e_{l,j}}\to\dots\to
\bigoplus_{j=0}^{m}P_{j}^{\oplus e_{0,j}}\to V\to 0\]
where $e_{0,j}=\dim V^j$ for all $j=0,\dots,m$
and, for any $l\geqq 1$ and any $j\leqq m-l$, $e_{l,j}$ is determined inductively by the following formula:
$e_{l,j}=\sum_{j<k}e_{l-1,k}\dim \Hom(G_j,G_k)$.
\end{lemma}
\begin{rmk}\label{MinimalResol}
The resolution above is not minimal in general.
Throughout this paper, 
we shall denote by 
\[0\to P_0^{\oplus e_{m,0}'}\to\dots\to \bigoplus_{j=0}^{m-l}P_{j}^{\oplus e_{l,j}'}\to\dots\to
\bigoplus_{j=0}^{m}P_{j}^{\oplus e_{0,j}'}\to V\to 0\]
a minimal resolution of $V$
with $0\leqq e_{l,j}'\leqq e_{l,j}$ for all $0\leqq j\leqq m-l\leqq m$.
\end{rmk}

\begin{prop}
\label{general resolution}
Under the assumption and notation as above,
let $F$ be a coherent sheaf on $X$.
Suppose that 
$\Ext^q(G,F)=0$ for all $q>0$.
Then $F$ has a locally free resolution of the following form:
\[0\to G_0^{\oplus e_{m,0}}\to\dots\to \bigoplus_{j=0}^{m-l}G_{j}^{\oplus e_{l,j}}\to\dots\to
\bigoplus_{j=0}^{m}G_{j}^{\oplus e_{0,j}}\to F\to 0\]
where $e_{0,j}=\dim \Hom(G_j,F)$ for all $j=0,\dots,m$
and, for any $l\geqq 1$ and any $j\leqq m-l$, $e_{l,j}$ is determined inductively by the following formula:
\[e_{l,j}=\sum_{j<k}e_{l-1,k}\dim \Hom(G_j,G_k).\]
\end{prop}
\begin{proof}
Since $\Ext^q(G,F)=0$ for all $q>0$,
Bondal's theorem \cite[Theorem 6.2]{MR992977} implies that  $\Hom(G, F)\lotimes_A G\cong F$.
Since $\Hom(G,F)^j\cong \Hom(G_j,F)$,
Lemma~\ref{projresol} shows 
the following projective resolution of $\Hom (G,F)$:
\[0\to P_0^{\oplus e_{m,0}}\to\dots\to \bigoplus_{j=0}^{m-l}P_{j}^{\oplus e_{l,j}}\to\dots\to
\bigoplus_{j=0}^{m}P_{j}^{\oplus e_{0,j}}\to \Hom(G,F)\to 0.\]
Since $P_j\otimes_AG\cong G_j$, the projective resolution above
induces the desired locally free resolution
of the coherent sheaf $F$.
\end{proof}

Set $\mathcal{P}_l=\oplus_{j=0}^{m-l}G_{j}^{\oplus e_{l,j}}$
for $0\leqq l\leqq m$,
and let $\mathcal{P}_{\bullet}$ denote the resulting complex.

\begin{rmk}
Set 
$\HN_i(\mathcal{P}_{\bullet})=
\oplus_{j=i}^{m-\bullet}
G_j^{\oplus e_{\bullet,j}}$.
Then $\mathcal{P}_{\bullet}$ has the following filtration of Harder-Narasimhan type
\[0\to \HN_m(\mathcal{P}_{\bullet})\to \dots \to \HN_i(\mathcal{P}_{\bullet})\to \dots\to \HN_0(\mathcal{P}_{\bullet})=\mathcal{P}_{\bullet}\]
with 
$\HN_{i}(\mathcal{P}_{\bullet})/\HN_{i+1}(\mathcal{P}_{\bullet})\cong 
\oplus_{i\leqq m-\bullet}G_i^{\oplus e_{\bullet,i}}\in \langle G_i\rangle
$,
where $\langle G_i\rangle$ denotes the smallest triangulated subcategory of $D^b(X)$ containing $G_i$ 
and closed under isomorphism.
Note that $\Hom_{D^b(X)}(\langle G_i\rangle,\langle G_j\rangle)=0$ if $i>j$.
If we regard $\mathcal{P}_{\bullet}$ and $F$ as objects of $D^b(X)$
and the filtration above as that in $D^b(X)$, 
then $\mathcal{P}_{\bullet}\cong F$ and semiorthogonality of 
$\langle G_0\rangle,\dots,\langle G_m\rangle$ implies that 
the filtration above is unique and functorial with respect to $F$.
\end{rmk}
Let $\mathcal{E}$ be a coherent sheaf on $X$,
and let $\mathcal{O}_X(1)$ be an ample line bundle on $X$.
By Serre's vanishing theorem,
we see that if $d'\gg 0$ then $\Ext^q(G,\mathcal{E}(d'))=0$ for all $q>0$.
Let $d_{\min}$ be the minimal integer $d_{\min}$ 
such that 
$\Ext^q(G,\mathcal{E}(d'))=0$ for all $q>0$ and all $d'\geqq d_{\min}$.

\begin{cor}\label{Resol}
Under the notation above, $\mathcal{E}(d_{\min})$ fits in the following exact sequence:
\[0\to G_0^{\oplus e_{m,0}}\to\dots\to \bigoplus_{j=0}^{m-l}G_{j}^{\oplus e_{l,j}}\to\dots\to
\bigoplus_{j=0}^{m}G_{j}^{\oplus e_{0,j}}\to \mathcal{E}(d_{\min})\to 0,\]
where $e_{0,j}=\dim \Hom(G_j,\mathcal{E}(d_{\min}))$ for all $j=0,\dots,m$ and,
for any $l\geqq 1$ and any $j\leqq m-l$, $e_{l,j}$ is determined inductively by 
$e_{l,j}=\sum_{j<k}e_{l-1,k}\dim \Hom(G_j,G_k)$.
Moreover we can replace $e_{l,j}$ by some integer $e_{l,j}'$
such that $0\leqq e_{l,j}'\leqq e_{l,j}$ for all $0\leqq j\leqq m-l\leqq m$ 
corresponding to the minimal resolution (see Remark~\ref{MinimalResol}
for the precise definition of $e_{l,j}'$).
\end{cor}
In the rest of this paper, 
we call the exact sequence in Corollary~\ref{Resol}
\textit{the standard resolution of $\mathcal{E}(d_{\min})$
with respect to the (prescribed) full strong exceptional sequence $(G_0,\dots,G_m)$
of vector bundles},
and let $e_{l,j}'$ and $e_{l,j}$ be as in Corollary~\ref{Resol}.
The following is an easy but fundamental relation among $e_{l,j}$'s.
\begin{prop}\label{easyConst} 
If $e_{l,k}=0$ for all $k>j$, then $e_{l+1,k}=0$ for all $k>j-1$.
\end{prop}
\begin{proof}
This follows immediately from the definition of $e_{l,j}$.
\end{proof}

To reduce $e_{0,0}$ to $e_{0,0}'$, the following proposition is also fundamental.
\begin{prop}\label{secondconstraint}
If $\mathcal{E}(d_{\min})$ does not admit 
$G_0$
as a quotient.
Then $e_{1,0}\geqq e_{0,0}$ and,
in the 
standard 
resolution, we can replace $\mathcal{P}_0$
by $\mathcal{P}_0^0$ 
and 
$\mathcal{P}_1$ by $\mathcal{P}_1^0$,
where 
$\mathcal{P}_0^0=\oplus_{j=1}^{m}G_{j}^{\oplus e_{0,j}}$
and $\mathcal{P}_1^0=G_0^{\oplus e_{1,0}-e_{0,0}}
\oplus(\oplus_{j=1}^{m-1}G_{j}^{\oplus e_{1,j}})$.
In particular we see that $e_{0,0}'=0$.
\end{prop}
\begin{proof}
Denote by $d_1$ the differential $\mathcal{P}_1\to \mathcal{P}_0$,
and let $p_{l,0}:\mathcal{P}_l\to G_0^{\oplus e_{l,0}}$ be the projection.
Then the composite $p_{0,0}\circ d_1$ factors through $p_{1,0}$:
$p_{0,0}\circ d_1=d_{1,0}\circ p_{1,0}$ 
where $d_{1,0}:G_0^{\oplus e_{1,0}}\to G_0^{\oplus e_{0,0}}$
is the induced morphism.
Suppose that $d_{1,0}$ is not surjective.
Then 
we have a surjection
$q:G_0^{\oplus e_{0,0}}\to G_0$ such that $q\circ d_{1,0}=0$.
Since $(q\circ p_{0,0})\circ d_1=q\circ d_{1,0}\circ p_{1,0}=0$,
there exists a surjection $\mathcal{E}(d_{\min})\to G_0$,
which contradicts the assumption.
Therefore $d_{1,0}$ is surjective and $e_{1,0}\geqq e_{0,0}$.
Moreover
$p_{0,0}\circ d_1
$ is surjective.
Since 
$\Ker(p_{0,0})=
\mathcal{P}_0^0
$ and 
$\Ker(p_{0,0}\circ d_1)
=\mathcal{P}_1^0$,
the desired replacement can be done in the 
standard
resolution of $\mathcal{E}(d_{\min})$.
\end{proof}
\begin{rmk}
If $d_{\min}>0$, $G_0=\mathcal{O}$, and $\mathcal{E}$ is nef,
then $\mathcal{E}(d_{\min})$ does not admit the sheaf $\mathcal{O}$
as a quotient.
\end{rmk}

\section{Some easy constraints}\label{easyconstraints}
Let $X$, $G_0,\dots,G_m$,
$\mathcal{E}$, $\mathcal{O}_X(1)$,
$d_{\min}$, $e_{l,j}$, $e_{l,j}'$, and 
$\mathcal{P}_l
$
be as in 
\S~\ref{Preliminaries}
for $0\leqq j\leqq l\leqq m$.
In the rest of this paper, we always assume that $\mathcal{E}$ is a nef vector bundle
of rank $r$
and that $\mathcal{O}_X(1)$ is a ``suitable" ample line bundle on $X$,
e.g., an ample line bundle of ``minimal degree".
Then it is natural to consider the following
\begin{Problem} 
What constraints does the condition that $\mathcal{E}$ is nef
impose on (or among) $d_{\min}$, $e_{l,j}$'s, and $e_{l,j}'$'s ?
Find good constraints on (or among) them.
\end{Problem}
The rest of this paper addresses this problem in the following cases:
\begin{enumerate}
\item $X$ is a projective space $\mathbb{P}^n$, 
$\mathcal{O}_X(1)$ is the ample generator of $\Pic X$,
and $(G_0,\dots, G_m)$ is equal to $(\mathcal{O},\mathcal{O}(1),\dots,\mathcal{O}(n))$;
\item $X$ is an odd dimensional smooth hyperquadric 
$\mathbb{Q}^n$ with $n\geqq 3$,
the field $K$ is of characteristic zero,
$\mathcal{O}_X(1)$ is the ample generator of $\Pic X$,
and $(G_0,\dots, G_m)$ is equal to $(\mathcal{O},\mathcal{S},\mathcal{O}(1),\dots,
\mathcal{O}(n-1)
)$,
where 
$\mathcal{S}$ is the spinor bundle on 
$\mathbb{Q}^{
n}$;
\item $X$ is an even dimensional smooth hyperquadric 
$\mathbb{Q}^{
n}$, 
the field $K$ is of characteristic zero,
and $(G_0,\dots, G_m)$ is equal to 
$(\mathcal{O},\mathcal{S}^+,\mathcal{S}^-, \mathcal{O}(1),\dots,
\mathcal{O}(
n-1))$,
where 
$\mathcal{S}^+$ and $\mathcal{S}^-$ are the spinor bundles on 
$\mathbb{Q}^{
n}$.
If 
$n\geqq 4$,
then $\mathcal{O}_X(1)$ is the ample generator of $\Pic X$,
and if 
$n=2$,
then $\mathcal{O}_X(1)$ is the ample line bundle $\mathcal{O}(1,1)$ of minimal degree.
\end{enumerate}
Thus we always assume, in the rest of the paper,  that if 
$X$ is as in (1), (2), or (3), then $(G_0,\dots,G_m)$ is as in (1), (2), or (3) respectively.

If $\Pic X\cong \mathbb{Z}$, 
we denote by $d$ the 
integer 
such that $\mathcal{O}_X(d)\cong \det\mathcal{E}$,
and if 
$X\cong \mathbb{Q}^2$,
we denote by $(a,b)$ the pair of 
integers 
such that $\mathcal{O}_{\mathbb{Q}^2}(a,b)\cong \det\mathcal{E}$.

In this section, we give some easy constraints among $d_{\min}$ and $e_{l,j}$'s in the cases 
(1), (2), and (3) above.

\begin{prop}\label{firstConstraint} 
We have the following constraints
\begin{enumerate} 
\item[(1)] Suppose that $X=\mathbb{P}^n$.
\begin{enumerate}
\item If $j>d+d_{\min}$, then $e_{0,j}=0$. 
\item If $d<r$, then $d_{\min}\geqq 0$.
\end{enumerate}
\item[(2)] Suppose that $X=\mathbb{Q}^{n}$.
\begin{enumerate}
\item Suppose that $n$ is odd. If $j>d+d_{\min}+1$, then $e_{0,j}=0$. 
\item Suppose that $n$ is even. If $j>d+d_{\min}+2$, then $e_{0,j}=0$. 
\item Suppose that $n$ is even. If $e_{l,k}=0$ for all $k>2$, then $e_{l+1,k}=0$ for all $k>0$.
In particular $e_{n,1}=0$ and $e_{n+1,0}=0$.
\item If $d<r$, then $d_{\min}\geqq 0$.
\item Suppose that $n=2$. If $\min \{a,b\}<r$, then $d_{\min}\geqq 0$.
\end{enumerate}
\end{enumerate}
\end{prop}
\begin{proof}
The proofs of (1) (a), (2) (a), and (2) (b) are essentially the same; the only difference comes from the fact
that there exists $\mathcal{S}$ or a pair of $\mathcal{S}^+$ and $\mathcal{S}^-$ between $\mathcal{O}$ and $\mathcal{O}(1)$,
so that the numbering of $G_j$ differs.
For simplicity, we only write explicitly the proof of (1) (a), but the reader will easily modify the proof for (2) (a), and (2) (b).

(1) (a) Suppose that 
$e_{0,j}\neq 0$ for some $j> d+d_{\min}$.
Then we have a non-zero map $\mathcal{O}(j)\to \mathcal{E}(d_{\min})$,
which gives a non-zero map $\mathcal{O}_L(j)\to \mathcal{E}|_L(d_{\min})$
for a general line $L$ in $\mathbb{P}^n$.
This implies that the maximal degree of a direct
summand of $\mathcal{E}|_L(d_{\min})$
is at least $j$.
On the other hand, since $\mathcal{E}$ is nef, 
the maximal degree of a direct
summand of $\mathcal{E}|_L(d_{\min})$ is at most $d+d_{\min}$.
This is a contradiction, since $j> d+d_{\min}$.

(2) (c) This is because 
$e_{l+1,1}=\sum_{1<k}e_{l,k}\dim \Hom(G_1,G_k)=e_{l,2}\dim \Hom(\mathcal{S}^+,\mathcal{S}^-)$
and $\Hom(\mathcal{S}^+,\mathcal{S}^-)=0$.

The proofs of (1) (b), (2) (d), and (2) (e) are essentially the same; 
For simplicity, we only write explicitly the proof of (2) (d).

(2) (d) We have a surjection $\mathcal{P}_0\to \mathcal{E}(d_{\min})$.
Since $\mathcal{P}_0$ is globally generated,
the restriction $\mathcal{E}|_L(d_{\min})$ to a line $L$ in $\mathbb{Q}^n$ is also globally generated.
Hence the minimal degree of a direct summand of $\mathcal{E}|_L(d_{\min})$ is non-negative.
Note here that the minimal degree of a direct summand 
of $\mathcal{E}|_L(d_{\min})$ 
is $d_{\min}$
since $r>d$.
Therefore $d_{\min}\geqq 0$.
\end{proof}

\section{An upper bound for $d_{\min}$}\label{UpperBound}
Let $X$, $G_0,\dots,G_m$,
$G$,
$\mathcal{E}$, $\mathcal{O}_X(1)$, and $d_{\min}$ be as in 
\S~\ref{Preliminaries}.
Assume that $\mathcal{E}$ be a nef vector bundle of rank $r$ as in \S~\ref{easyconstraints}.
In this section, we assume that the base field $K$ is of characteristic zero,
and we give an upper bound for $d_{\min}$ in the cases (1), (2), and (3) as described in \S~\ref{easyconstraints}.
Let $\mathcal{O}_X(1)$, $d$, and $(a,b)$ be as in \S~\ref{easyconstraints}.

Let $\pi:\mathbb{P}(\mathcal{E})\to X$ be the projection, and let $H(\mathcal{E})$ be the tautological line bundle
on $\mathbb{P}(\mathcal{E})$.

In Lemma~\ref{KodairaOnProjectiveSpace} below,
in order to unify and shorten descriptions in $\mathbb{P}^n$ and $\mathbb{Q}^n$ ($n\geqq 3$),
we 
denote by $c_1(X)$ the integer corresponding to $c_1(X)$
via the isomorphism $\Pic X\cong \mathbb{Z}$ which sends the ample generator to $1$.
However, needless to say, we must not use this abuse of notation in intersection formulas.

\begin{lemma}\label{KodairaOnProjectiveSpace}
Let $X$ be a smooth Fano variety of dimension $n$ with $\Pic X\cong \mathbb{Z}$,
and let $\mathcal{O}_X(1)$ be the ample generator of $\Pic X$.
Then we have the following vanishing.
\begin{enumerate}
\item[(1)] $\Ext^q(\mathcal{O}(j), \mathcal{E}(d))=0$ for all $q>0$ and $j< c_1(X)$.
\item[(2)] $\Ext^q(\mathcal{O}(c_1(X)), \mathcal{E}(d))=0$ for all $q>0$ if $H(\mathcal{E})^{n+r-1}>0$.
If $n=2$ then the condition $H(\mathcal{E})^{n+r-1}>0$ is equivalent to the one that $c_1(\mathcal{E})^2-c_2(\mathcal{E})>0$.
\end{enumerate}
\end{lemma}
\begin{proof}
(1)  We have isomorphisms 
\[\Ext^q(\mathcal{O}(j),\mathcal{E}(d))
\cong H^q(X,\mathcal{E}(d-j))
\cong H^q(\mathbb{P}(\mathcal{E}), H(\mathcal{E})+\pi^*\mathcal{O}_{X}(d-j)).\]
We claim that the last cohomology group vanishes by the Kodaira vanishing theorem;
indeed,
since $-K_X\cong \mathcal{O}_X(c_1(X))$,
we have 
\[H(\mathcal{E})+\pi^*\mathcal{O}_{X}(d-j)-K_{\mathbb{P}(\mathcal{E})}
\cong (r+1)H(\mathcal{E})+\pi^*\mathcal{O}_{X}(c_1(X)-j),\]
and $(r+1)H(\mathcal{E})+\pi^*\mathcal{O}_{X}(c_1(X)-j)$ is ample by the Nakai-Moishezon criterion, since $j<c_1(X)$.
Therefore the claim follows.

(2) If $j=c_1(X)$, then 
$H(\mathcal{E})+\pi^*\mathcal{O}_{\mathbb{P}^n}(d-j)-K_{\mathbb{P}(\mathcal{E})}$ 
is isomorphic to $(r+1)H(\mathcal{E})$, and this is nef and big by assumption.
The result then follows from the Kawamata-Viehweg vanishing theorem.
The assertion for $n=2$ follows from  $H(\mathcal{E})^{r+1}=c_1(\mathcal{E})^2-c_2(\mathcal{E})$ if $n=2$.
\end{proof}

\begin{cor}\label{ProSpaceResol}
Let $\mathcal{E}$ be a nef vector bundle of rank $r$ on $\mathbb{P}^n$.
Then $d_{\min}\leqq d$.
Moreover if $H(\mathcal{E})^{n+r-1}>0$ then $d_{\min}< d$.
In particular if $n=2$ and $c_1(\mathcal{E})^2-c_2(\mathcal{E})>0$,
then $d_{\min}< d$.
\end{cor}

\begin{lemma}\label{KodairaOnQuadric}
Suppose that $X$ is a smooth hyperquadric $\mathbb{Q}^n$ of dimension $n\geqq 3$.
Then 
$\Ext^q(\mathcal{S},\mathcal{E}(d+j))=0$ for all $q>0$ and $j\geqq -\lfloor\frac{n}{2}\rfloor+1$,
where $\mathcal{S}$ is a 
spinor bundle on $\mathbb{Q}^n$.
\end{lemma}
\begin{proof}
We have an isomorphism $\Ext^q(\mathcal{S},\mathcal{E}(d+j))\cong 
H^q(\mathbb{Q}^n, \mathcal{S}^{\vee}\otimes\mathcal{E}(d+j))$.
By Theorem~\ref{Usefulottaviani} (1), (2), and (3),
to show $H^q(\mathbb{Q}^n,\mathcal{S}^{\vee}\otimes\mathcal{E}(d+j))=0$
for all $q>0$ and $j\geqq -\lfloor\frac{n}{2}\rfloor+1$ and 
a 
spinor bundle $\mathcal{S}$ on $\mathbb{Q}^n$,
it is enough to show that $H^q(\mathbb{Q}^n,\mathcal{S}\otimes\mathcal{E}(d+j))=0$
for all $q>0$ and $j\geqq -\lfloor\frac{n}{2}\rfloor$ and 
a 
spinor bundle $\mathcal{S}$ on $\mathbb{Q}^n$.
We have an isomorphism $H^q(\mathbb{Q}^n,\mathcal{S}\otimes\mathcal{E}(d+j))\cong 
H^q(\mathbb{P}(\mathcal{S}), H(\mathcal{S})\otimes p^*(\mathcal{E}(d+j)))$,
where $p:\mathbb{P}(\mathcal{S})\to \mathbb{Q}^n$ is the projection
and $H(\mathcal{S})$ is the tautological line bundle
on $\mathbb{P}(\mathcal{S})$.
Let $\tilde{\pi}:\mathbb{P}(p^*\mathcal{E})\to \mathbb{P}(\mathcal{S})$ be the projection.
Since $H(\mathcal{S})\otimes p^*(\mathcal{E}(d+j))\cong 
p^*\mathcal{E}\otimes H(\mathcal{S})\otimes p^*\mathcal{O}_{\mathbb{Q}^n}(d+j)$,
we have an isomorphism
\[H^q(\mathbb{P}(\mathcal{S}), H(\mathcal{S})\otimes p^*(\mathcal{E}(d+j)))
\cong
H^q(\mathbb{P}(p^*\mathcal{E}), H(p^*\mathcal{E})
+
\tilde{\pi}^*(H(\mathcal{S})
+
p^*\mathcal{O}_{\mathbb{Q}^n}(d+j))).\]
We claim here that the last cohomology group vanishes by the Kodaira vanishing theorem;
first observe that 
$H(p^*\mathcal{E})+
\tilde{\pi}^*(H(\mathcal{S})+ p^*\mathcal{O}_{\mathbb{Q}^n}(d+j))-K_{\mathbb{P}(p^*\mathcal{E})}
$ is isomorphic to 
\[(r+1)H(p^*\mathcal{E})+
\tilde{\pi}^*(H(\mathcal{S})+ p^*\mathcal{O}_{\mathbb{Q}^n}(j)-K_{\mathbb{P}(\mathcal{S})}).\]
To show the last line bundle is ample, it is enough by the Nakai-Moishezon criterion to show that 
$H(\mathcal{S})+ p^*\mathcal{O}_{\mathbb{Q}^n}(j)-K_{\mathbb{P}(\mathcal{S})}$ is ample.
To see this, recall that $\mathbb{P}(\mathcal{S})$ is a flag manifold
parameterizing flags of one-dimensional and maximal dimensional linear subspaces
of $\mathbb{Q}^n$; set $s=\lfloor \frac{n}{2}\rfloor$, and let $q:\mathbb{P}(\mathcal{S})\to S$ be the projection,
which is a $\mathbb{P}^{s}$-bundle, to the spinor variety $S$.
Recall also that $H(\mathcal{S})\cong q^*\mathcal{O}_S(1)$
for the ample generator $\mathcal{O}_S(1)$ of $\Pic S$ (see, e.g., \cite[\S 5]{MR3275418}).
We see that $\mathbb{P}(\mathcal{S})$ is a Fano manifold of Picard number two,
that $H(\mathcal{S})+ p^*\mathcal{O}_{\mathbb{Q}^n}(j)-K_{\mathbb{P}(\mathcal{S})}$ is $p$-ample,
and that it is $q$-ample if $j+s\geqq 0$. Therefore $H(\mathcal{S})+ p^*\mathcal{O}_{\mathbb{Q}^n}(j)-K_{\mathbb{P}(\mathcal{S})}$
is ample if $j\geqq -s$, and the claim follows.
\end{proof}

\begin{cor}\label{quadricResol3}
Let $\mathcal{E}$ be a nef vector bundle of rank $r$ on a smooth hyperquadric $\mathbb{Q}^n$
of dimension $n\geqq 3$.
Then $d_{\min}\leqq d$.
Moreover if $n\geqq 4$ and $H(\mathcal{E})^{n+r-1}>0$ then $d_{\min}< d$.
\end{cor}
\begin{proof}
This follows from Lemmas~\ref{KodairaOnProjectiveSpace} and \ref{KodairaOnQuadric}.
\end{proof}

Finally we deal with the case $X=\mathbb{Q}^2$.
We have $\Ext^q(G,\mathcal{E}(d',d'))=0$ for $d'\gg 0$ and $q>0$ 
by Serre's vanishing theorem.
Note that if $(d_1,d_2)$ is a pair of integers such that $\Ext^q(G,\mathcal{E}(d_1,d_2))=0$ for all $q>0$,
then $\mathcal{E}(d_1,d_2)$ has the standard resolution with respect to $\mathcal{O}$, $\mathcal{O}(1,0)$,
$\mathcal{O}(0,1)$, and $\mathcal{O}(1,1)$, which implies that $\Ext^q(G,\mathcal{E}(d_1',d_2'))=0$ for all $q>0$,
all $d_1'\geqq d_1$, and all $d_2'\geqq d_2$.
Then we define a pair $(d_{1,\min},d_{2,\min})$ of integers
by the following property:
\[
\begin{split}
\Ext^q(G,\mathcal{E}(d_1',d_2'))&=0\textrm{ for all }q>0,\textrm{ all }d_1'\geqq d_{1,\min},\textrm{ and all }d_2'\geqq d_{2,\min},\\
\Ext^q(G,\mathcal{E}(d_{1,\min}-1,d_{2,\min}))&\neq 0\textrm{ for some }q>0,\\
\Ext^q(G,\mathcal{E}(d_{1,\min},d_{2,\min}-1))&\neq 0\textrm{ for some }q>0.
\end{split}
\]

\begin{lemma}\label{KodairaOnQuadricSurf}
Suppose that $X=\mathbb{Q}^2$.
Let $\mathcal{S}$ be a 
spinor bundle $\mathcal{O}(1,0)$ or $\mathcal{O}(0,1)$.
Then we have the following vanishing.
\begin{enumerate}
\item[(1)] $\Ext^q(\mathcal{O}(j,j), \mathcal{E}(a,b))=0$ for all $q>0$ and $j< 2$.
\item[(2)] $\Ext^q(\mathcal{O}(2,2), \mathcal{E}(a,b))=0$ for all $q>0$, if $2ab>c_2(\mathcal{E})$.
\item[(3)] $\Ext^q(\mathcal{S},\mathcal{E}(a+j,b+j))=0$ for all $q>0$ and $j\geqq 0$.
\item[(4)] $\Ext^q(\mathcal{S}(1,1),\mathcal{E}(a,b))=0$ for all $q>0$, if $2ab>c_2(\mathcal{E})$.
\end{enumerate}
\end{lemma}
\begin{proof}
(1) We have isomorphisms 
\[\Ext^q(\mathcal{O}(j,j),\mathcal{E}(a,b))
\cong H^q(\mathcal{E}(a-j,b-j))
\cong H^q(\mathbb{P}(\mathcal{E}), H(\mathcal{E})+\pi^*\mathcal{O}_{\mathbb{Q}^2}(a-j,b-j)).\]
We claim that the last cohomology group vanishes by the Kodaira vanishing theorem;
indeed we have 
\[H(\mathcal{E})+\pi^*\mathcal{O}_{\mathbb{Q}^2}(a-j,b-j)-K_{\mathbb{P}(\mathcal{E})}
\cong (r+1)H(\mathcal{E})+\pi^*\mathcal{O}_{\mathbb{Q}^2}(2-j,2-j)\]
and $(r+1)H(\mathcal{E})+\pi^*\mathcal{O}_{\mathbb{Q}^2}(2-j,2-j)$ is ample by the Nakai-Moishezon criterion since $j<2$.
Therefore the claim follows.

(2) Note that $H(\mathcal{E})^{r+1}=c_1(\mathcal{E})^2-c_2(\mathcal{E})=2ab-c_2(\mathcal{E})$.
Therefore if $2ab>c_2(\mathcal{E})$ then $H(\mathcal{E})$ is nef and big.
Hence $H(\mathcal{E})+\pi^*\mathcal{O}_{\mathbb{Q}^2}(a-j,b-j)-K_{\mathbb{P}(\mathcal{E})}$ is nef and big if $j=2$.
The result then follows from the Kawamata-Viehweg vanishing theorem.

(3) 
Suppose that $\mathcal{S}\cong \mathcal{O}(1,0)$.
We have isomorphisms 
\[\begin{split}
\Ext^q(\mathcal{S},\mathcal{E}(a+j,b+j))&\cong 
H^q(\mathbb{Q}^2, \mathcal{E}(a+j-1,b+j))\\
&\cong
H^q(\mathbb{P}(\mathcal{E}), H(\mathcal{E})
+
\pi^*\mathcal{O}_{\mathbb{Q}^2}(a+j-1,b+j)).
\end{split}\]
We show that the last cohomology group vanishes by the Kodaira vanishing theorem;
we see that $H(\mathcal{E})+
\pi^*\mathcal{O}_{\mathbb{Q}^2}(a+j-1,b+j)-K_{\mathbb{P}(\mathcal{E})}
$ is isomorphic to 
$(r+1)H(\mathcal{E})+
\pi^*\mathcal{O}_{\mathbb{Q}^2}(j+1,j+2)$,
and this line bundle is ample if $j\geqq 0$ by the Nakai-Moishezon criterion.

(4) The proof is almost the same as (3); if $j=-1$, then 
$(r+1)H(\mathcal{E})+
\pi^*\mathcal{O}_{\mathbb{Q}^2}(j+1,j+2)$ is nef and big if so is $H(\mathcal{E})$. 
Now the result follows from the Kawamata-Viehweg vanishing theorem.
\end{proof}

\begin{cor}\label{quadricResol}
Let $\mathcal{E}$ be a nef vector bundle of rank $r$ on $\mathbb{Q}^2$.
Then we can take 
$(d_{1,\min},d_{2,\min})$ such that 
$d_{1,\min}\leqq a$ and $d_{2,\min}\leqq b$.
Moreover we can take 
$(d_{1,\min},d_{2,\min})$ such that 
$d_{1,\min}\leqq a-1$ and $d_{2,\min}\leqq b-1$,
if $2ab>c_2(\mathcal{E})$.
\end{cor}

\section{Maximal degree subbundles of a nef vector bundle
}
Let $X$ be as in \S~\ref{Preliminaries},
and let $\mathcal{E}$
be as in \S~\ref{Preliminaries}.
Assume that $\mathcal{E}$ is a nef vector bundle of rank $r$ as in \S~\ref{easyconstraints}.

\begin{lemma}\label{detsub}
Suppose that there exists a non-zero morphism $\varphi:\det\mathcal{E}\to \mathcal{E}$. 
Then $\varphi$ makes $\det\mathcal{E}$ a subbundle of $\mathcal{E}$.
\end{lemma}
\begin{proof}
Let $s$ be a non-zero element of $H^0(\mathcal{E}\otimes (\det\mathcal{E})^{\vee})$ corresponding to $\varphi$,
and suppose that the zero locus $(s)_0$ of $s$ is not empty.
Take a curve $C$ such that $C\cap (s)_0\neq \emptyset$
and that $C$ is not contained in $(s)_0$,
and let $\pi:\tilde{C}\to C$ be the normalization.
Then $\mathcal{O}_{\tilde{C}}(\pi^*((s)_0\cap C))$ is a subbundle of 
$
\pi^*(\mathcal{E}\otimes (\det\mathcal{E})^{\vee})$.
This implies that $\pi^*\mathcal{E}$ has a quotient bundle of negative degree,
which contradicts that $\mathcal{E}$ is nef.
Therefore $(s)_0$ is empty and $\varphi$ makes $\det\mathcal{E}$ a subbundle of $\mathcal{E}$.
\end{proof}

\begin{prop}\label{dsubExist}
Suppose 
that $H^1(\det\mathcal{E})=0$
and 
that every nef vector bundle 
with 
trivial determinant 
is isomorphic to a direct sum of copies of $\mathcal{O}$.
Then $\Hom(\det\mathcal{E},\mathcal{E})\neq 0$ implies that $\mathcal{E}\cong \mathcal{O}^{\oplus r-1}\oplus\det\mathcal{E}$.
\end{prop}
\begin{proof}
If $\Hom(\det\mathcal{E},\mathcal{E})\neq 0$, 
then, by Lemma~\ref{detsub}, there exists an exact sequence
\[0\to \det\mathcal{E}\to \mathcal{E}\to \mathcal{F}\to 0\]
with $\mathcal{F}$ a vector bundle.
Since $\det\mathcal{F}\cong \mathcal{O}$, the assumption implies that
$\mathcal{F}\cong \mathcal{O}^{\oplus r-1}$.
Since $H^1(\det\mathcal{E})=0$, this implies that 
$\mathcal{E}\cong \mathcal{O}^{\oplus r-1}\oplus\det\mathcal{E}$.
\end{proof}

\begin{rmk}
The assumption of Proposition~\ref{dsubExist}
is satisfied if $X$ is either a projective space $\mathbb{P}^n$ or a hyperquadric $\mathbb{Q}^n$.
See, e.g., \cite[Chap. 1 Theorem 3.2.1]{oss} and \cite[Lemma 3.6.1]{w3}
\end{rmk}

\begin{lemma}\label{quotientTorsionFree}
Let $\mathcal{F}$ be a locally free coherent sheaf,
$\mathcal{G}$ a torsion-free coherent sheaf,
and let 
\[0\to \mathcal{F}\to \mathcal{G}\to \mathcal{H}\to 0\]
be an exact sequence of coherent sheaves on $X$.
If the support $Z$ of torsion subsheaf $\mathcal{T}$ of $\mathcal{H}$ has codimension $\geqq 2$ in $X$,
then $\mathcal{T}=0$, i.e., $\mathcal{H}$ is torsion-free.
\end{lemma}
\begin{proof}
Set $U=X\setminus Z$, and let $i:U\to X$ be the inclusion.
Let 
$\varphi:\mathcal{G}\to \mathcal{H}/\mathcal{T}$ be the composite of the two quotients
$\mathcal{G}\to \mathcal{H}$ and $\mathcal{H}\to \mathcal{H}/\mathcal{T}$.
Let $\mathcal{K}$ be the kernel of $\varphi$. We have the following exact sequence by 
the snake lemma.
\[0\to \mathcal{F}\xrightarrow{\psi} \mathcal{K}\to \mathcal{T}\to 0\]
Since the support of $\mathcal{T}$ is outside $U$,
we see that $\mathcal{F}|_U\cong \mathcal{K}|_U$. 
Since $\mathcal{K}$ is a subsheaf of a torsion-free sheaf $\mathcal{G}$,
$\mathcal{K}$ is torsion-free. Hence the canonical morphism $\mathcal{K}\to i_*(\mathcal{K}|_U)$ is injective.
On the other hand, we have isomorphisms $\mathcal{F}\cong i_*(\mathcal{F}|_U)\cong i_*(\mathcal{K}|_U)$.
Therefore $\psi$ is an isomorphism, and thus $\mathcal{T}\cong 0$. Hence $\mathcal{H}$ is torsion-free.
\end{proof}

Recall here that $\det:K(X)\to \Pic X$ is 
defined since $X$ is smooth and projective
and thus every coherent sheaf admits a finite locally free resolution.
Here $K(X)$ denotes the Grothendieck group of $X$.

In the rest of this section, we assume 
that $\Pic X\cong \mathbb{Z}$,
and let $\mathcal{O}_X(1)$ denote the ample generator of $\Pic X$.
Let $d$ be the integer such that $\det\mathcal{E}\cong \mathcal{O}_X(d)$.
\begin{lemma}\label{nefquotient}
If $\mathcal{G}$ is a quotient coherent sheaf of $\mathcal{E}$,
then $\det\mathcal{G}$ is nef.
Moreover if $\det\mathcal{G}\cong \mathcal{O}_X$
then the support of the torsion subsheaf of $\mathcal{G}$ has codimension $\geqq 2$ in $X$.
\end{lemma}
\begin{proof}
First suppose that $\mathcal{G}$ is torsion free,
and let $Z$ be the singular locus of $\mathcal{G}$,
i.e., the locus where $\mathcal{G}$ is not locally free. Then $Z$ has codimension $\geqq 2$.
Set $U=X\setminus Z$, and let $i:U\to X$ be the inclusion.
Observe that $\det \mathcal{G}$ is equal to the sheaf $i_*(\det(\mathcal{G}|_U))$.
Let $s$ be the rank of $\mathcal{G}$.
Then the surjection $\mathcal{E}\to \mathcal{G}$ induces a morphism $\wedge^{s}\mathcal{E}\to \det\mathcal{G}$
which is surjective on $U$.
Suppose, to the contrary, that $\det\mathcal{G}$ is not nef.
Then $\det\mathcal{G}$ is isomorphic to $\mathcal{O}(k)$ for some negative integer $k$
since $\Pic X\cong \mathbb{Z}$.
Let $C$ be a general smooth curve that intersect with $U$.
Then the restriction $\wedge^{s}\mathcal{E}|_C\to \det\mathcal{G}|_C\cong \mathcal{O}_C(k)$ is non-zero,
and the image of this morphism is a line bundle of negative degree on $C$.
This contradicts that $\wedge^{s}\mathcal{E}|_C$ is nef.
Therefore $\det\mathcal{G}$ is nef.

Now consider the general case. Let $\mathcal{T}$ be the torsion subsheaf of $\mathcal{G}$.
Then $\det(\mathcal{G}/\mathcal{T})$ is nef by the consideration above.
Since $\det\mathcal{G}\cong \det\mathcal{T}\otimes \det(\mathcal{G}/\mathcal{T})$,
it is enough to show that $\det\mathcal{T}$ is nef.
Suppose, for a moment, that $\mathcal{T}\cong \mathcal{O}_D(u):=\mathcal{O}_X(u)\otimes\mathcal{O}_D$ 
for some closed subvariety $D$ of $X$
and an integer $u$.
If $D$ has codimension $\geqq 2$ in $X$, then $\det\mathcal{T}=\mathcal{O}_X$.
If $D$ has codimension one, then $D$ is an ample Cartier divisor since $\Pic X\cong \mathbb{Z}$.
Thus $\det\mathcal{T}$ is isomorphic to an ample line bundle $\mathcal{O}_X(D)$.
Therefore $\det\mathcal{T}$ is nef if $\mathcal{T}\cong \mathcal{O}_D(u)$.
Now, for a general $\mathcal{T}$, recall that $\mathcal{T}$ has an ``irreducible decomposition'',
i.e., a filtration every graded piece of which is of the form $\mathcal{O}_D(u)$ for 
some 
integer $u$
where $D$ is a closed subvariety 
defined by an associated point of $\mathcal{T}$.
Since the assertion holds for every graded piece, we conclude that $\det\mathcal{T}$ is nef.

Finally if $\det\mathcal{G}\cong \mathcal{O}_X$ it follows from the consideration above 
that the support of $\mathcal{T}$ has codimension $\geqq 2$ in $X$.
\end{proof}

\begin{rmk}\label{rmk6.6}.
If $\dim X=2$ and $\mathcal{G}$ is torsion-free,  or if $\dim X=1$,
then the assumption that $\Pic X\cong \mathbb{Z}$
is unnecessary in Lemma~\ref{nefquotient}.
\end{rmk}

\begin{rmk}
If $\mathcal{G}$ is not locally free, then $\det\mathcal{G}\neq \wedge^s\mathcal{G}$ in general
where $s=\rk \mathcal{G}$.
For example, if $\dim X=1$ and $\mathcal{G}=\mathcal{O}^{\oplus 2}\oplus k(p)$
where $k(p)$ is the residue field at a point $p\in X$, then $\det\mathcal{G}\cong \mathcal{O}(p)$
whereas $\wedge^2\mathcal{G}\cong \mathcal{O}\oplus k(p)^{\oplus 2}$.
If $X=\mathbb{P}^2$ and $\mathcal{G}\cong \mathcal{O}\oplus \mathfrak{m}$
where $\mathfrak{m}$ is the maximal ideal of a point $p\in X$,
then $\det\mathcal{G}\cong \mathcal{O}$
whereas $\wedge^2\mathcal{G}\cong \mathfrak{m}$.
\end{rmk}

\begin{prop}\label{generalO(1)double}
Suppose that 
$H^0(X,\mathcal{O}_X(1))\neq 0$
and 
that $\dim \Hom(\mathcal{O}(d), \mathcal{E})=0$.
If $\dim \Hom(\mathcal{O}(d-1),\mathcal{E})\geqq 2$,
then $d\leqq 2$.
Moreover if $d=2$ then $\dim \Hom(\mathcal{O}(1),\mathcal{E})= 2$
and 
we have an exact sequence 
\[0\to 
\Hom(\mathcal{O}(1),\mathcal{E})
\otimes \mathcal{O}(1)\to \mathcal{E}\to \mathcal{G}\to 0\]
with $\mathcal{G}$ a 
vector bundle.
\end{prop}
\begin{proof}
Let $\sigma$ be a non-zero element of $\Hom(\mathcal{O}(d-1),\mathcal{E})$,
and $s$ the corresponding element of $H^0(\mathcal{E}(1-d))$.
Since $\mathcal{E}$ is torsion-free, $\sigma$ is generically injective.
Moreover $\sigma$ is injective since $\mathcal{O}(1)$ is torsion-free.
Since $H^0(\mathcal{O}(1))\neq 0$,
we have an injection 
$\Hom(\mathcal{O}(i+1),\mathcal{E})\to 
\Hom(\mathcal{O}(i),\mathcal{E})$ for any integer $i$.
Since $\Hom(\mathcal{O}(d), \mathcal{E})=0$,
we have $\Hom(\mathcal{O}(i), \mathcal{E})=0$ for all $i\geqq d$.
Since $\mathcal{E}$ is locally free,
this implies that the zero locus $(s)_0$ of $s$ has codimension $\geqq 2$.
Define a coherent sheaf $\mathcal{F}$ by the following exact sequence
\[0\to \mathcal{O}(d-1)\xrightarrow{\sigma} \mathcal{E}\to \mathcal{F}\to 0.\]
Then $\mathcal{F}$ is locally free outside the zero locus $(s)_0$ of $s$.
Thus the support of torsion subsheaf of $\mathcal{F}$ is contained in $(s)_0$.
Hence $\mathcal{F}$ is torsion-free by Lemma~\ref{quotientTorsionFree}.

We have an exact sequence
\[0\to \Hom(\mathcal{O}(d-1),\mathcal{O}(d-1))\to \Hom(\mathcal{O}(d-1),\mathcal{E})\to \Hom(\mathcal{O}(d-1),\mathcal{F}).
\]
In particular, 
we see that the image of the map $\Hom(\mathcal{O}(d-1),\mathcal{E})\to\Hom(\mathcal{O}(d-1), \mathcal{F})$
has dimension $\geqq 1$.
Let $\tau$ be a non-zero element 
in the image.
Since $\mathcal{F}$ is torsion-free, $\tau$ is generically injective.
Moreover $\tau$ is injective since $\mathcal{O}(d-1)$ is torsion-free.
Define a coherent sheaf $\mathcal{G}$ by the following exact sequence
\[0\to \mathcal{O}(d-1)\xrightarrow{\tau} \mathcal{F}\to \mathcal{G}\to 0.\]
Let $V$ be the pull back of the one-dimensional subspace $K\tau$ generated by $\tau$
by the map $\Hom(\mathcal{O}(d-1),\mathcal{E})\to \Hom(\mathcal{O}(d-1),\mathcal{F})$.
Then $V$ has dimension two. 
By the snake lemma, we see that there exists the following exact sequence
\[0\to V\otimes \mathcal{O}(d-1)\to \mathcal{E}\to \mathcal{G}\to 0.\]
Hence $\det\mathcal{G}\cong \mathcal{O}_X(2-d)$.
Since $\det\mathcal{G}$ is nef by Lemma~\ref{nefquotient},
we conclude that $d\leqq 2$.

Suppose moreover that $d=2$. Then $\det\mathcal{G}\cong \mathcal{O}_X$.
Lemma~\ref{nefquotient} implies 
that the support of the torsion subsheaf of $\mathcal{G}$
has codimension $\geqq 2$ in $X$.
Then $\mathcal{G}$ is torsion-free  by Lemma~\ref{quotientTorsionFree}.
Next we show that $V=\Hom(\mathcal{O}(1),\mathcal{E})$.
Suppose, to the contrary, that $V\subsetneq \Hom(\mathcal{O}(1),\mathcal{E})$.
Let $\upsilon$ be an element of $\Hom(\mathcal{O}(1),\mathcal{E})\setminus V$.
Then, since $\mathcal{G}$ and $\mathcal{O}(1)$ are torsion-free,
$\upsilon$ defines an injective morphism $\mathcal{O}_X(1)\to \mathcal{G}$,
which implies that $\mathcal{G}$ has a quotient sheaf $\mathcal{H}$ with $\det\mathcal{H}\cong \mathcal{O}_X(-1)$.
On the other hand, since $\mathcal{H}$ is also a quotient sheaf of $\mathcal{E}$,
$\det\mathcal{H}$ is nef by Lemma~\ref{nefquotient}. This is a contradiction. Therefore 
$V=\Hom(\mathcal{O}(1),\mathcal{E})$.
Finally we show that $\mathcal{G}$ is a vector bundle.
Let $Z$ be the singular locus of $\mathcal{G}$. Since $\mathcal{G}$ is torsion-free,
$Z$ has codimension $\geqq 2$.
For any point $x$ of $X$, take a curve $C$ which contains $x$ and is not contained in $Z$.
Let $\tilde{C}\to C$ be the normalization.
Then $\mathcal{G}\otimes \mathcal{O}_{\tilde{C}}$ is generically free of rank $r-2$.
Thus we have an exact sequence
\[0\to \Hom(\mathcal{O}(1),\mathcal{E})\otimes \mathcal{O}_{\tilde{C}}(1)\to \mathcal{E}\otimes\mathcal{O}_{\tilde{C}}
\to \mathcal{G}\otimes\mathcal{O}_{\tilde{C}}\to 0.\]
Since $\det \Hom(\mathcal{O}(1),\mathcal{E})\otimes \mathcal{O}_{\tilde{C}}(1)
\cong \det \mathcal{E}\otimes\mathcal{O}_{\tilde{C}}$, we see that $\det(\mathcal{G}\otimes\mathcal{O}_{\tilde{C}})\cong 
\mathcal{O}_{\tilde{C}}$.
Then $\mathcal{G}\otimes\mathcal{O}_{\tilde{C}}$ is torsion-free by Lemma~\ref{nefquotient} and Remark~\ref{rmk6.6}.
Thus $\mathcal{G}\otimes\mathcal{O}_{\tilde{C}}$ is locally free,
and hence  $\mathcal{G}$ is locally free at $x$.
Therefore $\mathcal{G}$ is a vector bundle.
\end{proof}

\begin{cor}\label{O(1)double}
Suppose 
that 
every nef vector bundle on $X$
with 
trivial determinant 
is isomorphic to a direct sum of copies of $\mathcal{O}_X$,
that 
$H^0(X,\mathcal{O}_X(1))\neq 0$,
and 
that $H^1(X,\mathcal{O}(1))=0$.
If $d=2$, $\dim \Hom(\mathcal{O}(2), \mathcal{E})=0$,
and 
$\dim \Hom(\mathcal{O}(1),\mathcal{E})\geqq 2$,
then 
$\mathcal{E}\cong \mathcal{O}(1)^{\oplus 2}\oplus \mathcal{O}^{\oplus r-2}$.
\end{cor}
\begin{proof}
By Proposition~\ref{generalO(1)double}, $\dim \Hom(\mathcal{O}(1),\mathcal{E})=2$
and we have the following exact sequence 
\[0\to 
\Hom(\mathcal{O}(1),\mathcal{E})
\otimes \mathcal{O}(1)\to \mathcal{E}\to \mathcal{G}\to 0\]
with $\mathcal{G}$ a vector bundle.
We see that 
$\mathcal{G}$ is nef with $\det\mathcal{G}\cong \mathcal{O}_X$.
Therefore $\mathcal{G}\cong \mathcal{O}^{\oplus r-2}$ by assumption.
Since $H^1(X,\mathcal{O}(1))=0$, we conclude that $\mathcal{E}\cong \mathcal{O}(1)^{\oplus 2}\oplus \mathcal{O}^{\oplus r-2}$.
\end{proof}

\begin{rmk}
The assumption on $(X,\mathcal{O}(1))$ of Corollary~\ref{O(1)double}
is satisfied if 
$(X,\mathcal{O}(1))$ is either $(\mathbb{P}^n,\mathcal{O}(1))$ or $(\mathbb{Q}^n,\mathcal{O}(1))$.
\end{rmk}

\begin{rmk}
Set $X=\mathbb{P}^2$, and set $\mathcal{F}=\mathfrak{m}_{x,X}\otimes\mathcal{O}(1)$,
where $\mathfrak{m}_{x,X}$ is the ideal sheaf on $X$ of a point $x$ of $X$.
Then $\mathcal{F}$ is a torsion-free sheaf,
and we see that $\dim H^0(\mathcal{F})=2$ and that $\dim H^0(\mathcal{F}(-1))=0$.
Let $H$ be a line passing through $x$,
and let $t$ be the corresponding element of $H^0(\mathcal{F})$.
Set $U=X\setminus\{x\}$. Then $H$ is the closure of the zero locus $(t|_U)_0$
of the restriction $t|_U$ of $t$ to $U$. However the restriction $t|_H\in H^0(\mathcal{F}|_H)$ of $t$
to $H$ does not vanish.
The element $t|_H$ generates the torsion subsheaf of $\mathcal{F}|_H$.
\end{rmk}

\section{The case where $X$ is a projective space}
Let $X$, $G_0,\dots,G_m$,
$G$, $A$, 
$\mathcal{E}$, $\mathcal{O}_X(1)$,
$d_{\min}$, 
and 
$e_{l,j}$
be as in \S~\ref{Preliminaries}
for $0\leqq j\leqq l\leqq m$.
Assume that $\mathcal{E}$ be a nef vector bundle of rank $r$ as in \S~\ref{easyconstraints}.
In this section, we assume 
that $X$, $G_0,\dots,G_m$, and $\mathcal{O}_X(1)$ are as in the case (1) in \S~\ref{easyconstraints}.
Let $d$ be as in \S~\ref{easyconstraints}. 

\begin{lemma}\label{hyperplaneIsom}
The following holds.
\begin{enumerate}
\item[(1)] Let $H$ be a hyperplane of $\mathbb{P}^n$. Then $d_{\min}$ for $\mathcal{E}|_H$ with respect to 
$(\mathcal{O}_H,\dots,\mathcal{O}_H(n-1))$ is less than or equal to $d_{\min}$ for $\mathcal{E}$ with respect to 
$(\mathcal{O},\dots,\mathcal{O}(n))$.
\item[(2)] Suppose that $H^0(\mathcal{E}(k-3))=0$ and $d_{\min}\leqq k$. Then,
for any $l$-dimensional linear section $\mathbb{P}^l$ of $\mathbb{P}^n$ with $l\geqq 2$,
the restriction map
$H^0(\mathcal{E}(k-2))\to H^0(\mathcal{E}|_{\mathbb{P}^l}(k-2))$
is an isomorphism.
\end{enumerate}
\end{lemma}
\begin{proof}
(1) Set $G'=\oplus_{i=0}^{n-1} \mathcal{O}(i)$.
For any $d'\geqq d_{\min}$, we have the following distinguished triangle
\[\RHom(G'(1),\mathcal{E}(d'))
\to
\RHom(G',\mathcal{E}(d'))
\to
\RHom(G'|_H,\mathcal{E}|_H(d'))
\to.
\]
Since $\Ext^q(G'(1),\mathcal{E}(d'))=0$
and $\Ext^q(G',\mathcal{E}(d'))=0$
for $q>0$,
we obtain for $q>0$ that $\Ext^q(G'|_H,\mathcal{E}|_H(d'))=0$.
Therefore $d_{\min}$ for $\mathcal{E}|_H$ is less than or equal to $d_{\min}$ for $\mathcal{E}$.

(2) Suppose that $n\geqq 3$, and let $H$ be a hyperplane section of $\mathbb{P}^n$.
Since $n\geqq 3$ and $d_{\min}\leqq k$, we see that $H^1(\mathcal{E}(k-3))=0$,
Since $H^0(\mathcal{E}(k-3))=0$ by assumption,
we obtain $H^0(\mathcal{E}(k-2))\cong H^0(\mathcal{E}|_{H}(k-2))$.
Since the statement (1) holds, we now obtain the statement by induction.
\end{proof}

\begin{lemma}\label{s0linear}
Suppose that 
$\Hom(\mathcal{O}(d),\mathcal{E})=0$.
If there exists a non-zero element $s$ of 
$H^0(\mathcal{E}(1-d))$,
then 
the zero locus $(s)_0$ of $s$ is 
either empty or a (reduced) point.
Moreover if $(s)_0$ is a point, then $r\geqq n\geqq 2$.
\end{lemma}
\begin{proof}
Set $Z=(s)_0$.
Since $H^0(\mathcal{E}(-d))=0$, 
$Z$ has codimension $c\geqq 2$ in $\mathbb{P}^n$.
Suppose that 
$Z$
is not empty.

We show that the length of the non-empty intersection 
$Z\cap L$
of 
$Z$
and a line $L$ is one unless 
$Z\cap L=L$.
Let $L$ be a line such that 
$Z\cap L$
is a non-empty finite set,
and let $l$ be the length of 
$Z\cap L$;
$l$ is a positive integer. Then we have $\mathcal{O}_L(l)$ as a subbundle of $\mathcal{E}|_L(1-d)$.
Thus $\mathcal{E}|_L$ has $\mathcal{O}_L(l+d-1)$ as a direct summand. On the other hand,
since $\mathcal{E}$ is nef, the degree of a direct summand of $\mathcal{E}$ is at most $d$.
Thus $l=1$.
This implies that
$Z$
in $\mathbb{P}^n$ has no secant lines that is not contained in 
$Z$,
and hence we see that 
$Z$
is a linear subspace $\mathbb{P}^{n-c}$
as sets.
Moreover we see that 
$Z$
is reduced.
Indeed, let $p$ be a point of 
$Z$
and let $I$ be the ideal sheaf of 
$Z\cap \mathbb{A}^n$
in a linear affine open subset $\mathbb{A}^n$ containing $p$.
Let $(x_1,\dots,x_n)$ be the affine coordinates of $\mathbb{A}^n$.
We may assume that $p=(0,\dots,0)$ and that the radical $\sqrt{I}$ of $I$ is $(x_1,\dots,x_c)$.
Let $l$ be the minimal integer such that $x_1^l\in I$,
and let $L$ be a line in $\mathbb{P}^n$ defined as the closure of the affine line
defined by $(x_2,\dots, x_n)$.
Then 
$Z\cap L$
is a non-empty finite subscheme of length at least $l$.
Thus we see, by the same argument as above, that $l=1$.
Hence $x_1\in I$. By the same way, we see that $x_i\in I$ for all $1\leqq i\leqq c$.
Therefore we conclude that 
$Z$ is reduced and thus 
$Z$
is a linear subscheme $\mathbb{P}^{n-c}$
of $\mathbb{P}^n$.

Let $\mathcal{I}$ be the ideal sheaf of 
$Z$
in $\mathbb{P}^n$. Then the conormal bundle
$\mathcal{I}/\mathcal{I}^2$ is isomorphic to $\mathcal{O}_{Z}(-1)^{\oplus c}$.
On the other hand, we have a surjection $\mathcal{E}^{\vee}(d-1)\to \mathcal{I}$.
Suppose that $Z$ contains a line $L_0$. Then we have a surjection
\[\mathcal{E}^{\vee}|_{L_0}(d-1)\to \mathcal{O}_{L_0}(-1)^{\oplus c}.
\]
In particular, $\mathcal{E}^{\vee}|_{L_0}(d-1)$ has an $\mathcal{O}_{L_0}(-1)$ as a quotient.
This implies that $\mathcal{E}^{\vee}|_{L_0}(d-1)$ is isomorphic to 
$\mathcal{O}_{L_0}(-1)\oplus \mathcal{O}_{L_0}(d-1)^{\oplus r-1}$
since $\mathcal{E}$ is nef.
However this means that $\mathcal{E}^{\vee}|_{L_0}(d-1)$ cannot have $\mathcal{O}_{L_0}(-1)^{\oplus c}$
as a quotient since $d\geqq 1$ and $c\geqq 2$. This is a contradiction.
Therefore $Z$ is a point.
\end{proof}

In the rest of this section, we assume that the base field $K$ is of characteristic zero.
Theorem~\ref{d=1onProSpace} below is a part of \cite[Theorem 1]{pswnef},
and is also a consequence of \cite[IV-2.2 Proposition]{ellia}.
We give a different proof of this result based on our framework:
general restrictions on $e_{l,j}$'s and $d_{\min}$ obtained so far
enable us to prove this theorem immediately.

\begin{thm}\label{d=1onProSpace}
Suppose that $d=1$, i.e.,
that $\det\mathcal{E}\cong \mathcal{O}(1)$. Then $\mathcal{E}$ is isomorphic to either $\mathcal{O}(1)\oplus \mathcal{O}^{\oplus r-1}$
or $T_{\mathbb{P}^n}(-1)\oplus \mathcal{O}^{\oplus r-n}$.
\end{thm}
\begin{proof}
We see first that $d_{\min}\leqq 1$ by Corollary~\ref{ProSpaceResol}.
If $r=1$, then $\mathcal{E}\cong \mathcal{O}(1)$. Suppose that $r\geqq 2$. Then $d_{\min}\geqq 0$ by Proposition~\ref{firstConstraint} (1) (b).
Suppose that $\Hom (\mathcal{O}(1),\mathcal{E})\neq 0$.
Then $\mathcal{E}\cong \mathcal{O}(1)\oplus \mathcal{O}^{\oplus r-1}$
by Proposition~\ref{dsubExist}. Suppose that 
$\Hom (\mathcal{O}(1),\mathcal{E})=0$.
If $d_{\min}$ were zero,
the standard resolution of $\mathcal{E}$ with respect to $(\mathcal{O},\dots,\mathcal{O}(n))$ 
implies that $\mathcal{E}\cong \mathcal{O}^{\oplus r}$, which contradicts that $d=1$.
Therefore $d_{\min}=1$.
Then the standard resolution of $\mathcal{E}(1)$ modified by Proposition~\ref{secondconstraint} is 
\[0\to \mathcal{O}\to \mathcal{O}(1)^{\oplus r+1}\to \mathcal{E}(1)\to 0,\]
since $\det\mathcal{E}(1)\cong \mathcal{O}(r+1)$.
Then we see $r\geqq n$ and $\mathcal{E}\cong T_{\mathbb{P}^n}(-1)\oplus \mathcal{O}^{\oplus r-n}$,
because $\mathcal{E}$ is a vector bundle.
\end{proof}

\begin{thm}\label{oneO(d-1)}
Suppose that $\Hom(\mathcal{O}(d),\mathcal{E})=0$
and that $\Hom(\mathcal{O}(d-1),\mathcal{E})\neq 0$.
Then $\mathcal{E}$ satisfies one of the following:
\begin{enumerate}
\item $\mathcal{E}\cong \mathcal{O}(d-1)\oplus \mathcal{O}(1)\oplus \mathcal{O}^{\oplus r-2}$.
\item $\mathcal{E}$ fits in an exact sequence
\[
0\to \mathcal{O}(-1)\to \mathcal{O}(d-1)\oplus \mathcal{O}^{\oplus r}\to \mathcal{E}\to 0.
\]
\end{enumerate}
\end{thm}
\begin{proof}
Let $s$ be a non-zero element of $H^0(\mathcal{E}(1-d))$.
Let 
\[0\to \mathcal{O}(d-1)\to \mathcal{E}\to \mathcal{F}\to 0\]
be an exact sequence of coherent sheaves
defined by $s$.

If $(s)_0=\emptyset$, then $\mathcal{F}$ is a nef vector bundle with $\det\mathcal{F}\cong \mathcal{O}(1)$.
Theorem~\ref{d=1onProSpace} then implies that $\mathcal{E}$ satisfies (1) or (2) of the theorem.

If $(s)_0\neq \emptyset$, then $(s)_0$ is a point $z$ by Lemma~\ref{s0linear}.
Consider the projection from 
the point $z$.
By eliminating the indeterminacy, we get a morphism $f:Y\to \mathbb{P}^{n-1}$ where $\varphi:Y\to \mathbb{P}^n$
is the blowing-up 
at the point $z$.
Let $E$ be the exceptional divisor of $\varphi$.
We see that $f$ is a $\mathbb{P}^{1}$-bundle.
Then we get the following exact sequence
\[0\to\varphi^*\mathcal{O}(d-1)\otimes\mathcal{O}(E)\to \varphi^*\mathcal{E}\to \mathcal{G}\to 0\]
for a vector bundle $\mathcal{G}$ on $Y$.
We see that $\mathcal{G}|_F\cong \mathcal{O}_F^{\oplus r-1}$ for any fiber $F\cong \mathbb{P}^{1}$ of $f$.
Thus there exists a vector bundle $\mathcal{H}$ on $\mathbb{P}^{n-1}$ such that $\mathcal{G}\cong f^*\mathcal{H}$.
By restricting the exact sequence
\[0\to\varphi^*\mathcal{O}(d-1)\otimes\mathcal{O}(E)\to \varphi^*\mathcal{E}\to f^*\mathcal{H}\to 0\]
to the exceptional divisor $E$, we see that $\mathcal{H}$ fits in an exact sequence
\[0\to \mathcal{O}(-1)\to \mathcal{O}^{\oplus r}\to \mathcal{H}\to 0.
\]
Since $\mathcal{H}$ is a vector bundle on $\mathbb{P}^{n-1}$, 
we infer that $r\geqq n$ and that 
\[\mathcal{H}\cong T_{\mathbb{P}^{n-1}}(-1)\oplus \mathcal{O}^{\oplus r-n}.\]
Hence $\varphi^*\mathcal{E}$ has $\varphi^*\mathcal{O}^{\oplus r-n}$ as a direct summand,
and thus $\mathcal{E}$ has $\mathcal{O}^{\oplus r-n}$ as a direct summand. Therefore we have
\[\mathcal{E}\cong \mathcal{O}^{\oplus r-n}\oplus \mathcal{E}_0\]
for some nef vector bundle $\mathcal{E}_0$
of rank $n$ with $\det \mathcal{E}_0\cong \mathcal{O}(d)$.
We may assume that $d\geqq 2$ and that $s$ is a non-zero element of $H^0(\mathcal{E}_0(1-d))$.
Then we have an exact sequence
\[0\to \mathcal{O}(d-1)\to \mathcal{E}_0\to \mathcal{F}_0\to 0,\]
where $\mathcal{F}_0$ is a torsion-free coherent sheaf with $\det \mathcal{F}_0\cong \mathcal{O}(1)$.

We claim here that $\mathbb{P}(\mathcal{F}_0)$ is nonsingular, although $\mathcal{F}_0$ is not a vector bundle.
Let $\pi:\mathbb{P}(\mathcal{E}_0)\to \mathbb{P}^n$ be the projection.
It is clear that $\mathbb{P}(\mathcal{F}_0)\cap \pi^{-1}(\mathbb{P}^n\setminus \{z\})$ is nonsingular.
We may assume that, locally around $z$, the section $s$ of $\mathcal{E}_0(1-d)$ can be written as 
\[s=z_1e_1+z_2e_2+\dots+z_ne_n,\]
where $(z_1,\dots,z_n)$ is a local coordinate system around $z$ with $z=(0,\dots,0)$
and $(e_1,\dots,e_n)$ is a locally free basis of $\mathcal{E}_0(1-d)$ around $z$.
Regarding $(e_1;\cdots;e_n)$ as a homogeneous coordinate system on $\pi^{-1}(z)\cong \mathbb{P}^{n-1}$,
we see that $ds=dz_1e_1+dz_2e_2+\dots+dz_ne_n$ on the cover of the fiber $\pi^{-1}(z)$
and that $ds$ does not vanish on the fiber $\pi^{-1}(z)$.
Therefore we conclude that $\mathbb{P}(\mathcal{F}_0)$ is also nonsingular along the fiber $\pi^{-1}(z)$.

Now the Kodaira vanishing theorem implies that $d_{\min}$ for $\mathcal{F}_0$ is less than two
by the similar argument as in the proof of Lemma~\ref{KodairaOnProjectiveSpace} (1).

If $H^0(\mathcal{F}_0(-1))\neq 0$, then the similar argument as in the proof of 
Proposition~\ref{generalO(1)double} and Corollary~\ref{O(1)double}
implies that $\mathcal{E}_0\cong \mathcal{O}(d-1)\oplus \mathcal{O}(1)\oplus \mathcal{O}^{\oplus n-2}$.
Thus we obtain the case (1) of the theorem.

Suppose that $H^0(\mathcal{F}_0(-1))=0$.  
We claim here that $d_{\min}$ for $\mathcal{F}_0$ is one.
Indeed, if $\rk \mathcal{F}_0=1$, then we see that $\mathcal{F}_0\cong \mathfrak{m}_z(1)$,
where $\mathfrak{m}_z$ is the ideal sheaf of $z$,
and the claim follows.
If $\rk \mathcal{F}_0\geqq 2$,
then we first infer that $d_{\min}$ for $\mathcal{F}_0$ is greater than or equal to zero
by the similar argument as in the proof of Proposition~\ref{firstConstraint},
since $\mathcal{F}_0$ is a torsion-free quotient of a nef vector bundle $\mathcal{E}_0$.
If $d_{\min}$ for $\mathcal{F}_0$ were zero,
then the standard resolution of $\mathcal{F}_0$ shows that $\mathcal{F}_0$ would be
isomorphic to $\mathcal{O}^{\oplus n-1}$,
which contradicts the fact that $\det\mathcal{F}_0\cong \mathcal{O}(1)$.
Therefore we conclude that the claim holds.
Then the standard resolution of 
$\mathcal{F}_0(1)$ modified according to Proposition~\ref{secondconstraint} is 
\[
0\to \mathcal{O}\to \mathcal{O}(1)^{\oplus n}\to \mathcal{F}_0(1)\to 0,
\]
which implies that $\mathcal{E}$ is in the case (2) of the theorem.
\end{proof}

The following  is the main part of \cite[Theorem 1]{pswnef} of Peternell-Szurek-Wi\'{s}niewski.
Based on our framework, we give a different proof of this result.
See Remark~\ref{avoidmisunderstanding1} for the seeming difference of Theorem~\ref{d2OnProSpace}
and \cite[Theorem 1]{pswnef}.
\begin{thm}\label{d2OnProSpace}
Suppose that $d=2$, i.e., that $\det\mathcal{E}\cong \mathcal{O}(2)$. Then $\mathcal{E}$ satisfies one of the following:
\begin{enumerate}
\item[(1)] $\mathcal{E}\cong \mathcal{O}(2)\oplus\mathcal{O}^{\oplus r-1}.$
\item[(2)] $\mathcal{E}\cong \mathcal{O}(1)^{\oplus 2}\oplus\mathcal{O}^{\oplus r-2}.$
\item[(3)] $\mathcal{E}$ fits in an exact sequence
\[0\to\mathcal{O}(-1)\to\mathcal{O}(1)\oplus\mathcal{O}^{\oplus r}\to \mathcal{E}\to 0.\]
\item[(4)] $\mathcal{E}$ fits in an exact sequence
\[0\to\mathcal{O}(-1)^{\oplus 2}\to\mathcal{O}^{\oplus r+2}\to \mathcal{E}\to 0.\]
\item[(5)] $\mathcal{E}$ fits in an exact sequence
\[0\to\mathcal{O}(-2)\to\mathcal{O}^{\oplus r+1}\to \mathcal{E}\to 0.\]
\item[(6)] $n=3$ and $\mathcal{E}$ fits in an exact sequence
\[0\to\mathcal{O}(-2)\to\mathcal{O}(-1)^{\oplus 4}\to \mathcal{O}^{\oplus r+3}\to \mathcal{E}\to 0.\]
\end{enumerate}
\end{thm}
\begin{proof}
Suppose that $\Hom(\mathcal{O}(2),\mathcal{E})\neq 0$.
Then we see that $\mathcal{E}\cong \mathcal{O}(2)\oplus\mathcal{O}^{\oplus r-1}$
by Proposition~\ref{dsubExist}.
In the following, we assume that $\Hom(\mathcal{O}(2),\mathcal{E})=0$.
If 
$\dim \Hom(\mathcal{O}(1),\mathcal{E})\neq 0$,
then Theorem~\ref{oneO(d-1)} shows that 
$\mathcal{E}$ is either in the case (2) or in the case (3) of the theorem.
We assume that 
$\dim \Hom(\mathcal{O}(1),\mathcal{E})=0$ 
in the following.
If $d_{\min}\leqq 0$, then the standard resolution of $\mathcal{E}$ implies that 
$\mathcal{E}\cong \mathcal{O}^{\oplus r}$,
which contradicts the assumption $d=2$.
Thus $d_{\min}\geqq 1$.
Then we have $n\geqq 2$.
We also see $d_{\min}\leqq 2$ by Corollary~\ref{ProSpaceResol}.

Suppose that $d_{\min}=1$.
Then the standard resolution of $\mathcal{E}(1)$ modified according to Proposition~\ref{secondconstraint} together with $\mathcal{O}(-1)$-twist 
is 
\[0\to\mathcal{O}(-1)^{\oplus e_{1,0}-e_{0,0}}\to\mathcal{O}^{\oplus e_{0,1}}\to \mathcal{E}\to 0.\]
Since $d=2$, we have $e_{1,0}-e_{0,0}=2$, and thus $e_{0,1}=r+2$. This is the case (4) of the theorem.

In the following, we assume that $d_{\min}=2$. 
We shall apply to $\mathcal{E}(1)$ the Bondal spectral sequence \cite[Theorem 1]{MR3275418}
\[
E_2^{p,q}=\caltor_{-p}^A(\Ext^q(G,\mathcal{E}(1)),G)
\Rightarrow
E^{p+q}=
\begin{cases}
\mathcal{E}(1)& \textrm{if}\quad  p+q= 0\\
0& \textrm{if}\quad  p+q\neq 0.
\end{cases}
\]
So 
we first claim that, under the assumption that $d_{\min}$ for $\mathcal{E}$ is two, equality
\[H^q(\mathcal{E}(1-n))=
\begin{cases}
K& \textrm{if}\quad  q= n-1\\
0& \textrm{if}\quad  q>0\textrm{ and }q\neq n-1
\end{cases}
\]
holds unless $n=3$ and $H^1(\mathcal{E}(-2))\cong K$.

We shall prove this claim by induction on $n$
unless $n=3$ and $H^1(\mathcal{E}(-2))\cong K$.
Let $H$ be a hyperplane in $\mathbb{P}^n$.
We have an exact sequence
\[
0\to \mathcal{E}(1-n)\to \mathcal{E}(2-n)\to \mathcal{E}|_H(1-(n-1))\to 0.
\]
Since $d_{\min}=2$, we see that $H^q(\mathcal{E}(2-n))=0$ for all $q>0$
and that $H^q(\mathcal{E}(1-n))\neq 0$ for some $q>0$.

Suppose that $n=2$.
The Riemann-Roch formula for a vector bundle $\mathcal{E}$ of rank $r$ on $\mathbb{P}^2$
is
\[\chi(\mathcal{E}
)=
r+\frac{1}{2}d(d+3)
-c_2(\mathcal{E}).\]
Since $d=2$ by assumption, 
the above formula 
implies that 
$h^0(\mathcal{E})=\chi(\mathcal{E})=r+5-c_2(\mathcal{E})$.
Since $0\leqq H(\mathcal{E})^{r+1}=c_1(\mathcal{E})^2-c_2(\mathcal{E})$, we have $c_2(\mathcal{E})\leqq 4$.
Hence 
$h^0(\mathcal{E})
\geqq
r+1$.
Since $h^0(\mathcal{E}|_H)=r+2$, $h^0(\mathcal{E}(-1))=0$, and $h^q(\mathcal{E}|_H)=0$ for all $q>0$,
we see that $h^q(\mathcal{E}(-1))=0$ for all $q\geqq 2$,
that $h^1(\mathcal{E}(-1))=1$, and that $h^0(\mathcal{E})=r+1$.
Hence the claim holds for $n=2$.

Suppose that $n\geqq 3$ and that the claim 
holds
for $n-1$.
Since we have $H^q(\mathcal{E}(2-n))=0$ for all $q\geqq 0$, 
we see that 
\[H^q(\mathcal{E}(1-n))\cong H^{q-1}(\mathcal{E}|_H(1-(n-1))) \textrm{ for all }q\geqq 1.\]
The point here is to show that $d_{\min}$ for $\mathcal{E}|_H$ is two,
unless $n=3$ and $H^1(\mathcal{E}(-2))\cong K$.
Note first that $d_{\min}$ for $\mathcal{E}|_H$ is less than or equal to two
by Lemma~\ref{hyperplaneIsom} (1)
and that we have proved the theorem in case $d_{\min}\leqq 1$.
Now suppose that $d_{\min}$ for $\mathcal{E}|_H$ is less than one.
Then $\mathcal{E}|_H$ splits,
so that
$\mathcal{E}$ also splits by \cite[Chap. 1, Theorem 2.3.2]{oss}.
This contradicts that $d_{\min}=2$.
Suppose that $d_{\min}$ for $\mathcal{E}|_H$ is one.
Then 
$\mathcal{E}|_H$ does not split.
We have $H^{q-1}(\mathcal{E}|_H(1-(n-1)))=0$ for all $q\geqq 2$.
Hence $H^{q}(\mathcal{E}(1-n))=0$ for all $q\geqq 2$.
Since  $H^{q}(\mathcal{E}(1-n))\neq 0$ for some $q\geqq 1$,
we see that $H^1(\mathcal{E}(1-n))\neq 0$.
Thus $H^0(\mathcal{E}|_H(1-(n-1)))\neq 0$.
Hence $n=3$ since $\mathcal{E}|_H$ does not split.
Then $\mathcal{E}|_H$ is in the case (3) of the theorem,
and thus $H^0(\mathcal{E}|_H(-1))\cong K$.
Hence $H^1(\mathcal{E}(-2))\cong K$.
This shows that $d_{\min}$ for $\mathcal{E}|_H$ is two
unless $n=3$ and $H^1(\mathcal{E}(-2))\cong K$.
Now the claim holds by induction,
because $H^0(\mathcal{E}|_H(1-(n-1))=0$ if $d_{\min}$ for $\mathcal{E}|_H$ is two.

We shall show that $\mathcal{E}$ is in the case (5) of the theorem
unless $n=3$ and $H^1(\mathcal{E}(-2))\cong K$.
First, by the claim above and the assumption that $h^0(\mathcal{E}(-1))=0$, we see that 
\[\Ext^q(G,\mathcal{E}(1))=
\begin{cases}
\Hom(\mathcal{O},\mathcal{E}(1))\oplus\Hom(\mathcal{O}(1),\mathcal{E}(1))& \textrm{if}\quad  q=0\\
\Ext^{n-1}(G_n,\mathcal{E}(1))=K& \textrm{if}\quad  q= n-1\\
0& \textrm{if}\quad  q>0\textrm{ and }q\neq n-1.\\
\end{cases}
\]
Hence we have $E_2^{p,q}=0$ unless $q= n-1$ or $0$.
As we have shown in the proof of \cite[Proposition 1]{MR3275418},
we also infer that 
\[
E_2^{p,n-1}=\mathcal{H}^p(\Ext^{n-1}(G,\mathcal{E}(1))\lotimes_A G)
=
\begin{cases}
\mathcal{O}(-1)& \textrm{if}\quad  p= -n\\
0& \textrm{if}\quad  p\neq -n.
\end{cases}
\]
We finally see that a right $A$-module $\Hom(G,\mathcal{E}(1))$ has a projective resolution of the following form
\[
0\to P_0^{\oplus f_{1,0}}\to P_0^{\oplus f_{0,0}}\oplus P_1^{\oplus f_{0,1}}\to \Hom (G,\mathcal{E}(1))\to 0,
\]
where $P_0$ and $P_1$ are as in \S~\ref{Preliminaries},
$f_{0,j}=\dim \Hom(G_j, \mathcal{E}(1))$ ($j=0,1$), and $f_{1,0}=(n+1)f_{0,1}$.
Hence we see that
\[E_2^{p,0}=
\begin{cases}
\Ker (\mathcal{O}^{\oplus f_{1,0}}\to \mathcal{O}^{\oplus f_{0,0}}\oplus \mathcal{O}(1)^{\oplus f_{0,1}})& \textrm{if}\quad  p= -1\\
\Coker (\mathcal{O}^{\oplus f_{1,0}}\to \mathcal{O}^{\oplus f_{0,0}}\oplus \mathcal{O}(1)^{\oplus f_{0,1}})& \textrm{if}\quad  p= 0\\
0& \textrm{if}\quad p\neq -1, 0.
\end{cases}
\]
Thus we infer that
\[E_{\infty}^{p,q}=
\begin{cases}
E_n^{-n,n-1}=\Ker (\mathcal{O}(-1)\to E_2^{0,0})& \textrm{if}\quad  (p,q)= (-n,n-1)\\
E_n^{0,0}=\Coker (\mathcal{O}(-1)\to E_2^{0,0})& \textrm{if}\quad  (p,q)= (0,0)\\
E_2^{-1,0}& \textrm{if}\quad (p,q)=(-1,0)\\
0&\textrm{otherwise}.
\end{cases}
\]
The Bondal spectral sequence then shows that
\[E_{\infty}^{p,q}=
\begin{cases}
\mathcal{E}(1)& \textrm{if}\quad  (p,q)= (0,0)\\
0&\textrm{otherwise}.
\end{cases}
\]
This shows that we have exact sequences
\begin{gather*}
0\to \mathcal{O}(-1)\to E_2^{0,0}\to \mathcal{E}(1)\to 0,\\
0\to \mathcal{O}^{\oplus f_{1,0}}\to \mathcal{O}^{\oplus f_{0,0}}\oplus \mathcal{O}(1)^{\oplus f_{0,1}}\to E_2^{0,0}\to 0.
\end{gather*}
Therefore we get an exact sequence
\[
0\to \mathcal{O}(-1)\oplus \mathcal{O}^{\oplus f_{1,0}}\to \mathcal{O}^{\oplus f_{0,0}}\oplus \mathcal{O}(1)^{\oplus f_{0,1}}
\to \mathcal{E}(1)\to 0.
\]
Since $\det\mathcal{E}(1)\cong \mathcal{O}(r+2)$, we have $f_{0,1}=r+1$.
We claim here that the composite of the inclusion $\mathcal{O}^{\oplus f_{1,0}}\to \mathcal{O}(-1)\oplus \mathcal{O}^{\oplus f_{1,0}}$,
the morphism 
$\mathcal{O}(-1)\oplus \mathcal{O}^{\oplus f_{1,0}}\to \mathcal{O}^{\oplus f_{0,0}}\oplus \mathcal{O}(1)^{\oplus f_{0,1}}$,
and the projection $\mathcal{O}^{\oplus f_{0,0}}\oplus \mathcal{O}(1)^{\oplus f_{0,1}}\to \mathcal{O}^{\oplus f_{0,0}}$ is surjective.
Assume, to the contrary, that the composite is not surjective. Then there exists a surjection $\mathcal{O}^{\oplus f_{0,0}}\oplus \mathcal{O}(1)^{\oplus f_{0,1}}
\to \mathcal{O}$ such that the composite 
$\mathcal{O}^{\oplus f_{1,0}}\to \mathcal{O}(-1)\oplus \mathcal{O}^{\oplus f_{1,0}}\to \mathcal{O}^{\oplus f_{0,0}}\oplus \mathcal{O}(1)^{\oplus f_{0,1}}
\to \mathcal{O}$ is zero. 
The morphism $\mathcal{O}(-1)\oplus \mathcal{O}^{\oplus f_{1,0}}\to \mathcal{O}^{\oplus f_{0,0}}\oplus \mathcal{O}(1)^{\oplus f_{0,1}}
\to \mathcal{O}$ then induces a morphism $\mathcal{O}(-1)\to \mathcal{O}$,
whose quotient is either $\mathcal{O}_{\mathbb{P}^n}$ or $\mathcal{O}_H$ for some hyperplane $H$ in $\mathbb{P}^n$.
This implies that $\mathcal{E}(1)$ have $\mathcal{O}_{\mathbb{P}^n}$ or $\mathcal{O}_H$ as a quotient, 
which contradicts that $\mathcal{E}$ is nef.
Therefore the claim holds, and we can modify the sequence above to
\[
0\to \mathcal{O}(-1)\oplus \mathcal{O}^{\oplus f_{1,0}-f_{0,0}}\to \mathcal{O}(1)^{\oplus r+1}
\to \mathcal{E}(1)\to 0.
\]
By looking at ranks, we infer that $f_{1,0}-f_{0,0}=0$, and we get 
the case (5) of the theorem.

Finally suppose that $n=3$ and that $H^1(\mathcal{E}(-2))\cong K$.
Then as we have seen above, we may assume that $d_{\min}$ for $\mathcal{E}|_H$ is one.
We have a distinguished triangle
\[
\RHom(G,\mathcal{E}(1))\to \RHom(G,\mathcal{E}(2))\to \RHom(\oplus_{i=-1}^{n-1}\mathcal{O}_H(i),\mathcal{E}|_H(1))\to.
\]
Hence we see that $\Ext^{q}(G,\mathcal{E}(1))=0$ for all $q\geqq 2$.
Since $\Ext^1(G,\mathcal{E}(2))=0$, we have $\Ext^1(G,\mathcal{E}(1))=\Ext^1(G_n,\mathcal{E}(1))\cong K$.
Therefore we have 
\[\Ext^q(G,\mathcal{E}(1))=
\begin{cases}
\Hom(\mathcal{O},\mathcal{E}(1))\oplus\Hom(\mathcal{O}(1),\mathcal{E}(1))& \textrm{if}\quad  q=0\\
\Ext^{1}(G_n,\mathcal{E}(1))\cong K& \textrm{if}\quad  q= 1\\
0& \textrm{if}\quad  q\geqq 2.
\end{cases}
\]
Hence we have 
$E_2^{p,q}=0$ for all $q\geqq 2$.
By the same argument as 
in the proof of \cite[Proposition 1]{MR3275418},
we also infer that 
\[
E_2^{p,1}=\mathcal{H}^p(\Ext^1(G,\mathcal{E}(1))\lotimes_A G)
=
\begin{cases}
\mathcal{O}(-1)& \textrm{if}\quad  p= -3\\
0& \textrm{if}\quad  p\neq -3.
\end{cases}
\]
Finally 
a right $A$-module $\Hom(G,\mathcal{E}(1))$ has a projective resolution of the following form
\[
0\to P_0^{\oplus f_{1,0}}\to P_0^{\oplus f_{0,0}}\oplus P_1^{\oplus f_{0,1}}\to \Hom (G,\mathcal{E}(1))\to 0,
\]
where $P_0$ and $P_1$ are as in \S~\ref{Preliminaries},
$f_{0,j}=\dim \Hom(G_j, \mathcal{E}(1))$ ($j=0,1$), and $f_{1,0}=4f_{0,1}$.
Hence we see that
\[E_2^{p,0}=
\begin{cases}
\Ker (\mathcal{O}^{\oplus f_{1,0}}\to \mathcal{O}^{\oplus f_{0,0}}\oplus \mathcal{O}(1)^{\oplus f_{0,1}})& \textrm{if}\quad  p= -1\\
\Coker (\mathcal{O}^{\oplus f_{1,0}}\to \mathcal{O}^{\oplus f_{0,0}}\oplus \mathcal{O}(1)^{\oplus f_{0,1}})& \textrm{if}\quad  p= 0\\
0& \textrm{if}\quad p\neq -1, 0.
\end{cases}
\]
Thus we infer that
\[E_{\infty}^{p,q}=
\begin{cases}
E_3^{-3,1}=\Ker (\mathcal{O}(-1)\to E_2^{-1,0})& \textrm{if}\quad  (p,q)= (-3,1)\\
E_3^{-1,0}=\Coker (\mathcal{O}(-1)\to E_2^{-1,0})& \textrm{if}\quad  (p,q)= (-1,0)\\
E_2^{0,0}& \textrm{if}\quad (p,q)=(0,0)\\
0&\textrm{otherwise}.
\end{cases}
\]
The Bondal spectral sequence then shows that
\[E_{\infty}^{p,q}=
\begin{cases}
\mathcal{E}(1)& \textrm{if}\quad  (p,q)= (0,0)\\
0&\textrm{otherwise}.
\end{cases}
\]
Therefore $\mathcal{O}(-1)\cong E_2^{-1,0}$ and we get an exact sequence
\[
0\to \mathcal{O}(-1)\to \mathcal{O}^{\oplus f_{1,0}}\to \mathcal{O}^{\oplus f_{0,0}}\oplus \mathcal{O}(1)^{\oplus f_{0,1}}
\to \mathcal{E}(1)\to 0.\]
Since $\det\mathcal{E}(1)\cong \mathcal{O}(r+2)$, we see that $f_{0,1}=r+3$.
By looking at ranks, we also see that $f_{1,0}-f_{0,0}=4$.
Since $\mathcal{E}(1)$ does not admit $\mathcal{O}$ as a quotient, the sequence above can be replaced by the following one
\[
0\to \mathcal{O}(-1)\to \mathcal{O}^{\oplus 4}\to \mathcal{O}(1)^{\oplus r+3}
\to \mathcal{E}(1)\to 0.\]
This is the case (6) of the theorem.
\end{proof}
\begin{rmk}\label{avoidmisunderstanding1}
Note that $\Omega_{\mathbb{P}^3}(2)$ has the following locally free resolution 
\[0\to \mathcal{O}(-2)\to \mathcal{O}(-1)^{\oplus 4}\to \mathcal{O}^{\oplus 6}\to \Omega_{\mathbb{P}^3}(2)\to 0.\]
Thus if $\mathcal{E}$ on $\mathbb{P}^3$ fits in  the resolution 
$0\to \mathcal{O}\to \Omega_{\mathbb{P}^3}(2)\oplus \mathcal{O}^{\oplus r-2}\to \mathcal{E}\to 0$
given in \cite[Theorem 1 (2)]{pswnef}, then 
$\mathcal{E}$ also fits in  a resolution
\[0\to \mathcal{O}(-2)
\to \mathcal{O}(-1)^{\oplus 4}\oplus \mathcal{O}\to \mathcal{O}^{\oplus r+4}\to \mathcal{E}\to 0.\]
This implies that $\mathcal{E}$ fits in the resolution in the case (6) of Theorem~\ref{d2OnProSpace}.
Similarly, if $\mathcal{E}$ fits in the resolution given in \cite[Theorem 1 (3)]{pswnef}, then
it also fits in the resolution in the case (3) of Theorem~\ref{d2OnProSpace}.
See also \cite[\S 4, Proposition 1 and Remark 2]{MR3275418}.
\end{rmk}
\begin{rmk}
Suppose that $\mathcal{E}$ is in the case (6) of Theorem~\ref{d2OnProSpace}.
Since $H^1(\mathcal{E}(-2))\cong K$, $\mathcal{E}$ cannot split.
If $r\geqq 3$, then $h^0(\mathcal{E}^{\vee})\geqq r-3$ since $h^0(\Omega_{\mathbb{P}^3}(2))=6$.
Hence $\mathcal{E}\cong \mathcal{O}^{\oplus r-3}\oplus \mathcal{E}_0$
for some vector bundle $\mathcal{E}_0$ of rank three.
Note that $\mathcal{E}_0$ is also in the case (6) of Theorem~\ref{d2OnProSpace}.
Let $\mathcal{E}_0$ be a nef vector bundle in the case (6) of Theorem~\ref{d2OnProSpace}
and suppose that $\rk \mathcal{E}_0=3$. Then $c_3(\mathcal{E}_0)=0$,
and we see that $\mathcal{E}_0$ fits in an exact sequence
\[0\to \mathcal{O}\to \mathcal{E}_0\to \mathcal{F}\to 0,\]
where $\mathcal{F}$ is a nef vector bundle in the case (6) of Theorem~\ref{d2OnProSpace}.
Let $Z$ be the zero locus of a general element $s$ in $H^0(\mathcal{F})$.
Then $Z$ is a smooth curve of degree two, and we have an exact sequence
\[
0\to \mathcal{O}(-2)\to \mathcal{F}(-2)\to \mathcal{I}_Z\to 0,
\]
where $\mathcal{I}_Z$ is the ideal sheaf of $Z$ in $\mathbb{P}^3$.
Since $H^1(\mathcal{F}(-2))\cong K$, we have $H^1(\mathcal{I}_Z)\cong K$.
Hence $Z$ cannot be connected. Therefore $Z$ is a disjoint union of two lines.
Since $\Ext^1(\mathcal{I}_Z,\mathcal{O}(-2))\cong K$, we conclude that $\mathcal{F}
\cong \mathcal{N}(1)$, where $\mathcal{N}$ is a null correlation bundle.
Since $\Ext^1(\mathcal{F},\mathcal{O})\cong K$, we see that $\mathcal{E}_0$
is either $\Omega_{\mathbb{P}^3}(2)$ or $\mathcal{O}\oplus \mathcal{N}(1)$.
\end{rmk}

\section{The case where $X$ is a smooth quadric surface}
Let $X$, $G_0,\dots,G_m$, 
$G$,
$A$,
$\mathcal{E}$, $\mathcal{O}_X(1)$,
$d_{\min}$, 
and 
$e_{l,j}$
be as in \S~\ref{Preliminaries}
for $0\leqq j\leqq l\leqq m$.
Assume that $\mathcal{E}$ be a nef vector bundle of rank $r$ as in \S~\ref{easyconstraints}.
In this section, we assume 
that the base field $K$ is of characteristic zero
and that $X$, $G_0,\dots,G_m$, and $\mathcal{O}_X(1)$ are as in the case (3) in \S~\ref{easyconstraints}
with $n=2$.
In particular, $X$ is a smooth quadric surface $\mathbb{Q}^2$,
$m=3$, and $G_0,G_1,G_2,G_3$ are respectively $\mathcal{O},\mathcal{O}(1,0),\mathcal{O}(0,1), \mathcal{O}(1,1)$.
Let $(a,b)$ be as in \S~\ref{easyconstraints}. 

In Theorem~\ref{Chern1surface} below,
we classify
the above $\mathcal{E}$'s
with $\det\mathcal{E}\cong \mathcal{O}(1,1)$.
Note that such $\mathcal{E}$'s 
were already classified in \cite[\S 3]{swCompo}
and \cite[\S 2 Lemmas 1 and 2]{pswnef}
(see also Remark~\ref{errorstatement}).
We give a different proof 
of this result
in our framework.

\begin{thm}\label{Chern1surface}
Suppose that 
$(a,b)=(1,1)$, i.e.,
that $\det\mathcal{E}\cong \mathcal{O}(1,1)=\mathcal{O}(1)$.
Then 
$\mathcal{E}$
satisfies
one of the following:
\begin{enumerate}
\item $\mathcal{E}\cong\mathcal{O}^{\oplus r-1}\oplus\mathcal{O}(1)$.
\item $\mathcal{E}\cong \mathcal{O}^{\oplus r-2}\oplus\mathcal{O}(1,0)\oplus\mathcal{O}(0,1)$.
\item $\mathcal{E}$ fits in an exact sequence
\[0\to\mathcal{O}(-1)\to\mathcal{O}^{\oplus r+1}\to \mathcal{E}\to 0.\]
\end{enumerate}
\end{thm}
\begin{proof}
If $r=1$, then $\mathcal{E}\cong \mathcal{O}(1)$ and $d_{\min}=-1$. 
We assume that $r\geqq 2$ in the following.
Then $0\leqq d_{\min}\leqq 1$
by Proposition~\ref{firstConstraint} (2) (e)
and Corollary ~\ref{quadricResol}.

Since $\mathcal{E}$ is nef and $\det\mathcal{E}\cong\mathcal{O}(1)$,
we see that 
$\mathcal{E}|_L\cong \mathcal{O}_L^{\oplus r-1}\oplus \mathcal{O}_L(1)$
for any line $L$ in $\mathbb{Q}^2$. 

If $\Hom(\mathcal{O}(1),\mathcal{E})\neq 0$,
then it follows from Proposition~\ref{dsubExist} that 
$\mathcal{E}\cong \mathcal{O}^{\oplus r-1}\oplus\mathcal{O}(1)$.
In the following we assume that $\Hom(\mathcal{O}(1),\mathcal{E})= 0$.

Suppose that $\Hom(\mathcal{O}(0,1),\mathcal{E})\neq 0$.
Let $\varphi$ be a non-zero element of $\Hom(\mathcal{O}(0,1), \mathcal{E})$.
Since $\Hom(\mathcal{O}(1),\mathcal{E})= 0$,
$\varphi|_L\neq 0$ for any line $L$ of type $(1,0)$ in $\mathbb{Q}^2$.
Hence $\varphi|_L$ makes $\mathcal{O}(0,1)|_L$ a subbundle of $\mathcal{E}|_L$.
Therefore $\mathcal{O}(0,1)$ is a subbundle of $\mathcal{E}$ via $\varphi$.
Set $\mathcal{F}=\mathcal{E}/\mathcal{O}(0,1)$.
Then $\mathcal{F}$ is a nef vector bundle of rank $r-1$ with $\det\mathcal{F}\cong\mathcal{O}(1,0)$,
and thus $\mathcal{F}$ is isomorphic to $\mathcal{O}^{\oplus r-2}\oplus \mathcal{O}(1,0)$.
Therefore $\mathcal{E}\cong \mathcal{O}^{\oplus r-2}\oplus \mathcal{O}(1,0)\oplus\mathcal{O}(0,1)$.
Similarly if $\Hom(\mathcal{O}(1,0),\mathcal{E})\neq 0$ 
then $\mathcal{E}\cong \mathcal{O}^{\oplus r-2}\oplus \mathcal{O}(1,0)\oplus\mathcal{O}(0,1)$.

In the following, we assume that $\Hom(\mathcal{O}(0,1),\mathcal{E})= 0$ and that $\Hom(\mathcal{O}(1,0),\mathcal{E})= 0$.
Under these assumptions, we have $d_{\min}=1$. Indeed, if $d_{\min}$ were zero,
then $e_{0,3}=0$ by the assumption that $\Hom(\mathcal{O}(1),\mathcal{E})=0$,
and similarly $e_{0,2}=0$ and $e_{0,1}=0$ by the assumptions above.
The standard resolution then forces $\mathcal{E}$ to be isomorphic to $\mathcal{O}^{\oplus r}$,
which contradicts the assumption that $(a,b)=(1,1)$.
Therefore $d_{\min}=1$.

We shall apply to $\mathcal{E}$ the Bondal spectral sequence \cite[Theorem 1]{MR3275418}
\[
E_2^{p,q}=\caltor_{-p}^A(\Ext^q(G,\mathcal{E}),G)
\Rightarrow
E^{p+q}=
\begin{cases}
\mathcal{E}& \textrm{if}\quad  p+q= 0\\
0& \textrm{if}\quad  p+q\neq 0.
\end{cases}
\]
First note that $\Hom(G,\mathcal{E})\cong \Hom(\mathcal{O},\mathcal{E})\cong H^0(\mathcal{E})$.
The Riemann-Roch formula for a vector bundle $\mathcal{E}$ of rank $r$ on $\mathbb{Q}^2$
is
\[\chi(\mathcal{E})=c_1'(\mathcal{E})c_1''(\mathcal{E})-c_2(\mathcal{E})+c_1'(\mathcal{E})+c_1''(\mathcal{E})+r,\]
where $c_1(\mathcal{E})=(c_1'(\mathcal{E}), c_1''(\mathcal{E}))$.
Since $(c_1'(\mathcal{E}), c_1''(\mathcal{E}))=(1,1)$ by assumption, the above formula implies 
that $h^0(\mathcal{E})=\chi(\mathcal{E})=r+3-c_2(\mathcal{E})$.
Note here that $0\leqq H(\mathcal{E})^{r+1}=c_1(\mathcal{E})^2-c_2(\mathcal{E})$.
Hence we have $c_2(\mathcal{E})\leqq 2$, 
and consequently $h^0(\mathcal{E})\geqq r+1$.
We have an exact sequence
\[
0\to \mathcal{E}(-1,0)\to \mathcal{E}\to \mathcal{E}|_L\to 0,
\]
where $L$ is a line on $\mathbb{Q}^2$ of type $(1,0)$.
Thus we have an exact sequence
\[0\to H^0(\mathcal{E})\to H^0(\mathcal{E}|_L)\to H^1(\mathcal{E}(-1,0))\to 0
\]
by our assumption. Since $h^0(\mathcal{E}|_L)=r+1$, we infer that $h^0(\mathcal{E})=r+1$
and that $h^1(\mathcal{E}(-1,0))=0$.
Moreover we see that $\RHom(\mathcal{O}(1,0),\mathcal{E})\cong 0$, and similarly
we have $\RHom(\mathcal{O}(0,1),\mathcal{E})\cong 0$.
We have an exact sequence
\[
0\to \mathcal{E}(-1,-1)\to \mathcal{E}(0,-1)\to \mathcal{O}_L\oplus \mathcal{O}_L(-1)^{\oplus r-1}\to 0,
\]
and we see that $\RHom(\mathcal{O}(1,1),\mathcal{E})\cong K[-1]$.

Summing up, we have
\[\Ext^q(G,\mathcal{E})=
\begin{cases}
\Hom(\mathcal{O},\mathcal{E})& \textrm{if}\quad  q=0\\
\Ext^{1}(G_3,\mathcal{E})\cong K& \textrm{if}\quad  q= 1\\
0& \textrm{if}\quad  q= 2.
\end{cases}
\]
Hence we have 
$E_2^{p,q}=0$ for all $q\geqq 2$.
Let $S_k$ $(0\leqq k\leqq 3)$ be the right $A$-module corresponding to the representation such that 
$\Gr^jS_k=0$ for any $j\neq k$, $\Gr^kS_k=K$, and all the arrows are zero.
Then the right $A$-module $\Ext^1(G,\mathcal{E})$ is isomorphic to $S_3$.
Note here that 
\[S_3\lotimes_AG\cong \mathcal{O}(-1)[2],\]
since $\RHom(G,\mathcal{O}(-1)[2])\cong \Ext^2(G_3,\mathcal{O}(-1))\cong S_3$.
Hence we infer that 
\[\Ext^1(G,\mathcal{E})\lotimes_A G\cong \mathcal{O}(-1)[2].\] 
Therefore we see that 
\[
E_2^{p,1}=\mathcal{H}^p(\Ext^1(G,\mathcal{E})\lotimes_A G)
=
\begin{cases}
\mathcal{O}(-1)& \textrm{if}\quad  p= -2\\
0& \textrm{if}\quad  p\neq -2.
\end{cases}
\]
Finally 
a right $A$-module $\Hom(G,\mathcal{E})$ 
is isomorphic to a projective module $P_0^{\oplus f_{0,0}}$
where $P_0$ is as in \S~\ref{Preliminaries}
and $f_{0,0}=\dim \Hom(G_0, \mathcal{E}(1))$.
Hence we see that
\[E_2^{p,0}=
\begin{cases}
\mathcal{O}^{\oplus f_{0,0}}& \textrm{if}\quad  p= 0\\
0& \textrm{if}\quad p\neq 0.
\end{cases}
\]
Thus we infer that
\[E_{\infty}^{p,q}=
\begin{cases}
E_3^{-2,1}=\Ker (\mathcal{O}(-1)\to \mathcal{O}^{\oplus f_{0,0}})& \textrm{if}\quad  (p,q)= (-2,1)\\
E_3^{0,0}=\Coker (\mathcal{O}(-1)\to \mathcal{O}^{\oplus f_{0,0}})& \textrm{if}\quad  (p,q)= (0,0)\\
0&\textrm{otherwise}.
\end{cases}
\]
The Bondal spectral sequence then shows that
\[E_{\infty}^{p,q}=
\begin{cases}
\mathcal{E}& \textrm{if}\quad  (p,q)= (0,0)\\
0&\textrm{otherwise}.
\end{cases}
\]
Therefore we get an exact sequence
\[
0\to \mathcal{O}(-1)\to \mathcal{O}^{\oplus f_{0,0}}
\to \mathcal{E}\to 0.\]
By looking at ranks, we see that $f_{0,0}=r+1$,
and we get the case (3) of the theorem.
\end{proof}
\begin{rmk}\label{errorstatement}
In the statement of \cite[\S 2 Lemma 1]{pswnef} in case $(a,b)=(1,1)$, the case (3),
where $d_{\min}=1$, in Theorem~\ref{Chern1surface} is missing,
and, instead, ``the restriction of a spinor bundles from $\mathbb{Q}^3$" is added.
Since the restriction of a spinor bundles from $\mathbb{Q}^3$ is $\mathcal{O}(1,0)\oplus \mathcal{O}(0,1)$,
where $d_{\min}=0$,
this is an error. However one can understand that 
the case (3) in Theorem~\ref{Chern1surface}
would be 
what they actually wanted to say by the terms ``the restriction of a spinor bundles from $\mathbb{Q}^3$''
if one read through \cite[\S 3]{swCompo}.
\end{rmk}

\section{Results on spinor bundles}\label{Spinor bundles}
In this section, we assume that the base field $K$ is of characteristic zero,
and recall some results on spinor bundles.
Although we do not follow his convention for ``spinor bundles'', Ottaviani's results in ~\cite{ot}
is very useful in this paper. We rephrase his results under Kapranov's convention
for later use. 
Throughout this section,
let $\mathcal{S}$ denote the (spanned) spinor bundle on 
an odd-dimensional 
smooth hyperquadric
$\mathbb{Q}^{n}$,
and $\mathcal{S}^+$ and $\mathcal{S}^-$  the (spanned) spinor bundles on 
an even-dimensional 
smooth hyperquadric
$\mathbb{Q}^{n}$.
Besides that the sequences $(G_0,\dots, G_m)$ in the cases (2) and (3) of \S~\ref{easyconstraints}
are strong and exceptional,
all 
the results we need
about spinor bundles
are summarized in the following 
theorem.
\begin{thm}\label{Usefulottaviani}
Set $s=\lfloor \frac{n-1}{2}\rfloor$
and let $H$ be a smooth hyperplane section of $\mathbb{Q}^n$.
Then we have the following.
\begin{enumerate}
\item[(0)] $\mathcal{O}(1,0)$ and $\mathcal{O}(0,1)$
are (spanned) spinor bundles on $\mathbb{Q}^2$.
\item $\mathcal{S}^+|_{H}\cong \mathcal{S}$ and $\mathcal{S}^-|_{H}\cong \mathcal{S}$.
\item $H^0(\mathbb{Q}^n, \mathcal{S}^+)\cong H^0(H,\mathcal{S})$
and $H^0(\mathbb{Q}^n, \mathcal{S}^-)\cong H^0(H,\mathcal{S})$.
\item $\mathcal{S}|_{H}\cong \mathcal{S}^+\oplus\mathcal{S}^-$.
\item $H^0(\mathbb{Q}^n, \mathcal{S})\cong H^0(H,\mathcal{S}^+\oplus\mathcal{S}^-)$.
\item $\rk \mathcal{S}=2^s$, $\dim H^0(\mathcal{S})=2^{s+1}$, and $\det\mathcal{S}=\mathcal{O}(2^{s-1})$.
\item 
$\rk \mathcal{S}^{+}=2^s=\rk \mathcal{S}^{-}$ and $\dim H^0(\mathcal{S}^{+})=2^{s+1}=\dim H^0(\mathcal{S}^{-})$.
\item $\det\mathcal{S}^{+}=\mathcal{O}(2^{s-1})=\det\mathcal{S}^{-}$ if $s\geqq 1$.
\item $\mathcal{S}$, $\mathcal{S}^+$, and  $\mathcal{S}^-$ are all $\mu$-stable bundles (with respect to 
any ample line bundle).
\item $\mathcal{S}^{\vee}\cong \mathcal{S}(-1)$.
\item $(\mathcal{S}^+)^{\vee}\cong \mathcal{S}^+(-1)$ 
and $(\mathcal{S}^-)^{\vee}\cong \mathcal{S}^-(-1)$ 
and if $s$ is 
odd.
\item $(\mathcal{S}^+)^{\vee}\cong \mathcal{S}^-(-1)$ 
and $(\mathcal{S}^-)^{\vee}\cong \mathcal{S}^+(-1)$ 
and if $s$ is 
even.
\end{enumerate}
\end{thm}
\begin{proof}
Note that our spinor bundles are the duals of those of Ottaviani's.
The statement of (0) is,
e.g., in \cite[Example 1.5]{ot},
and already used in this paper.
(1) and (3) follow from \cite[Theorem 1.4]{ot}.
(2) and (4) follow from (1), (3), and \cite[Theorem 2.3]{ot} (or Bott's vanishing theorem).
(5) and (6) follow from (0), (1), (2), (3), and (4).
(7) follows from (0), (1), and (3).
A theorem of Ramanan 
\cite{MR0190947} 
and Umemura \cite[Theorem (2.4)]{MR0473243}
shows (8).
Finally (9), (10), and (11) follow from \cite[Theorem 2.8]{ot},
since $n=2(s+1)$ if $n$ is even.
\end{proof}

\begin{lemma}\label{S(1)toS}
We have the following isomorphisms.
\begin{enumerate}
\item[(1)] If $n$ is odd, then $\RHom(\mathcal{S}(1),\mathcal{S})\cong K[-1]$.
\item[(2)] If $n$ is even, then 
\begin{equation*}
\begin{split}
\RHom(\mathcal{S}^+(1),\mathcal{S}^+)\cong 0,\qquad &\RHom(\mathcal{S}^-(1),\mathcal{S}^-)\cong 0,\\
\RHom(\mathcal{S}^+(1),\mathcal{S}^-)\cong K[-1],\qquad &\RHom(\mathcal{S}^-(1),\mathcal{S}^+)\cong K[-1].
\end{split}
\end{equation*}
In particular, the following isomorphisms hold.
\[\RHom((\mathcal{S}^+\oplus\mathcal{S}^-)(1),\mathcal{S}^+)\cong K[-1],\textrm{\quad }
\RHom((\mathcal{S}^+\oplus\mathcal{S}^-)(1),\mathcal{S}^-)\cong K[-1]\]
\end{enumerate}
\end{lemma}
\begin{proof}
(1) Suppose that $n$ is odd.
Then 
$\mathcal{S}|_{H}\cong \mathcal{S}^+\oplus \mathcal{S}^-$
for a smooth hyperplane section $H$ of $\mathbb{Q}^n$
by Theorem~\ref{Usefulottaviani}.
We have the following distinguished triangle
\[
\RHom(\mathcal{S}(1),\mathcal{S})\to 
\RHom(\mathcal{S},\mathcal{S})\to
\RHom(\mathcal{S}^+\oplus \mathcal{S}^-, \mathcal{S}^+\oplus \mathcal{S}^-)
\to. 
\]
Since $\RHom(\mathcal{S},\mathcal{S})\cong K$
and $\RHom(\mathcal{S}^+\oplus \mathcal{S}^-, \mathcal{S}^+\oplus \mathcal{S}^-)\cong K\oplus K$,
we get a distinguished triangle
\[
\RHom(\mathcal{S}(1),\mathcal{S})\to 
K\to
K\oplus K
\to. 
\]
Since $\mathcal{S}$ is $\mu$-stable by Theorem~\ref{Usefulottaviani}, 
we have $\Hom(\mathcal{S}(1),\mathcal{S})=0$.
Therefore we conclude that $\RHom(\mathcal{S}(1),\mathcal{S})\cong K[-1]$.

(2) Suppose that $n$ is even.
Then $\mathcal{S}^+|_{H}\cong \mathcal{S}$
and $\mathcal{S}^-|_{H}\cong \mathcal{S}$
for a smooth hyperplane section $H$ of $\mathbb{Q}^n$
by Theorem~\ref{Usefulottaviani}.
We have the following distinguished triangle
\[
\RHom(\mathcal{S}^+(1),\mathcal{S}^+)\to 
\RHom(\mathcal{S}^+,\mathcal{S}^+)\to
\RHom(\mathcal{S}, \mathcal{S})
\to. 
\]
Since $K\cong \RHom(\mathcal{S}^+,\mathcal{S}^+)\to \RHom(\mathcal{S}, \mathcal{S})\cong K$
is isomorphic, we see that 
\[\RHom(\mathcal{S}^+(1),\mathcal{S}^+)\cong 0.\]
By the similar argument, we get $\RHom(\mathcal{S}^-(1),\mathcal{S}^-)\cong 0$.
We have the following distinguished triangle
\[
\RHom(\mathcal{S}^+(1),\mathcal{S}^-)\to 
\RHom(\mathcal{S}^+,\mathcal{S}^-)\to
\RHom(\mathcal{S}, \mathcal{S})
\to. 
\]
Since $\RHom(\mathcal{S}^+,\mathcal{S}^-)\cong 0$, we see that $\RHom(\mathcal{S}^+(1),\mathcal{S}^-)\cong K[-1]$.
By the similar argument, we get $\RHom(\mathcal{S}^-(1),\mathcal{S}^+)\cong K[-1]$
since $\RHom(\mathcal{S}^-,\mathcal{S}^+)\cong 0$.
\end{proof}

\section{The case where $X$ is a smooth hyperquadric}
Let $X$, $G_0,\dots,G_m$, 
$\mathcal{E}$, $\mathcal{O}_X(1)$,
$d_{\min}$, $e_{l,j}$, 
and 
$\mathcal{P}_l
$
be as in \S~\ref{Preliminaries}
for $0\leqq j\leqq l\leqq m$.
Assume that $\mathcal{E}$ be a nef vector bundle of rank $r$ as in \S~\ref{easyconstraints}.
In this section, we assume 
that the base field $K$ is of characteristic zero,
and that $X$, $G_0,\dots,G_m$, and $\mathcal{O}_X(1)$ are as in the cases (2) and (3) in \S~\ref{easyconstraints}
with $n\geqq 3$.
In particular, $X$ is a smooth hyperquadric $\mathbb{Q}^n$ of dimension $n\geqq 3$.
Let $d$ be as in \S~\ref{easyconstraints}.

\begin{lemma}\label{withoutO(1)}
Suppose 
that $d=1$ and that $\Hom (\mathcal{O}(1),\mathcal{E})=0$.
Let $\mathbb{Q}^2$ be a 
linear section 
of dimension two
of $\mathbb{Q}^n$.
Then 
$\Hom (\mathcal{O}_{\mathbb{Q}^2},\mathcal{E}|_{\mathbb{Q}^2})\cong \Hom (\mathcal{O},\mathcal{E})$.
\end{lemma}
\begin{proof}
We have a distinguished triangle
\[
\RHom(\mathcal{O}(k),\mathcal{E}(1))
\to
\RHom(\mathcal{O}(k-1),\mathcal{E}(1))
\to
\RHom(\mathcal{O}_H(k-1),\mathcal{E}|_H(1))
\to
\]
for a hyperplane section $H$ of $\mathbb{Q}^n$ and a integer $k$.
Since $\Hom (\mathcal{O}(2),\mathcal{E}(1))=0$ by assumption
and $d_{\min}\leqq 1$ by Corollary~\ref{quadricResol3}, 
we have $\RHom(\mathcal{O}(k),\mathcal{E}(1))=0$
for $2\leqq k\leqq n-1$.
Therefore $\RHom(\mathcal{O}(k-1),\mathcal{E}(1))
\cong
\RHom(\mathcal{O}_H(k-1),\mathcal{E}|_H(1))$ for $1\leqq k-1\leqq n-2$.
Hence we see that 
$\Hom(\mathcal{O}(1),\mathcal{E}(1))
\cong
\Hom(\mathcal{O}_H(1),\mathcal{E}|_H(1))$
and that if $n\geqq 4$
then $\Hom(\mathcal{O}_H(2),\mathcal{E}|_H(1))$ is zero.
Now we obtain the desired formulas by induction.
\end{proof}
\begin{prop}\label{dmin0spinor}
Suppose that $d=1$ and 
that $\Hom (\mathcal{O}(1),\mathcal{E})=0$.
If $\Hom (\mathcal{S},\mathcal{E})\neq 0$
for a spinor bundle $\mathcal{S}$,
then $n=3$ or $4$, and $\mathcal{E}\cong \mathcal{S}\oplus \mathcal{O}^{\oplus r-2}$.
\end{prop}
\begin{proof}
Let $\varphi:\mathcal{S}\to\mathcal{E}$ be a non-zero element of $\Hom (\mathcal{S},\mathcal{E})$.

Suppose that 
$\varphi|_{H}=0$ for some smooth hyperplane section $H$ of $\mathbb{Q}^n$.
Then we have 
$\Hom (\mathcal{S},\mathcal{E}(-1))\neq 0$;
let $\psi$ be a non-zero element of $\Hom (\mathcal{S},\mathcal{E}(-1))$.
We have $\psi|_{L}\neq 0$ for a general $2$-dimensional linear section $L$ of $\mathbb{Q}^n$.
Note
here that $\mathcal{S}|_L\cong (\mathcal{O}(1,0)\oplus \mathcal{O}(0,1))^{\oplus s}$,
where $s=\lfloor \frac{n-1}{2}\rfloor$,
by Theorem~\ref{Usefulottaviani}.
Hence $H^0((\mathcal{S}|_L)^{\vee}\otimes \mathcal{E}|_L(-1))=0$ by Theorem~\ref{Chern1surface}.
This contradicts the fact that $\psi|_L\neq 0$. 
Hence 
$\varphi|_{H}\neq 0$ for any smooth hyperplane 
section $H$ of $\mathbb{Q}^n$.
Since 
the restriction of a spinor bundle to a smooth hyperplane section
is again a spinor bundle or a direct sum of spinor bundles
by Theorem~\ref{Usefulottaviani},
the argument above implies, by induction,
that $\varphi|_{\mathbb{Q}^2}\neq 0$ for any $2$-dimensional smooth linear section
$\mathbb{Q}^2$ of $\mathbb{Q}^n$.

Denote by $\mathcal{Q}$ the image $\im (\varphi)$ of $\varphi$
and by $\mathcal{F}$ the cokernel $\Coker(\varphi)$ of $\varphi$.
Let $D$ be the singular locus of $\mathcal{F}$,
i.e., let its complement $X\setminus D$ be the set of points at which $\mathcal{F}$ is locally free.
Let $E$ be the singular locus of $\mathcal{Q}$.
Then $E$ is contained in $D$.
Since $\mathcal{Q}$ is torsion-free, $E$ has codimension $\geqq 2$.
Note that for each point $x$ in $\mathbb{Q}^n$ we can take a smooth $2$-dimensional linear section $L$
of $\mathbb{Q}^n$ such that $L$ contains $x$, that $L$ is not contained in $D$,
and that $L\cap E$ has codimension $\geqq 2$ in $L$.
We have a surjection $\mathcal{Q}|_L\to \im (\varphi|_L)$.
On the other hand, $\mathcal{Q}|_{L\setminus D}\to \mathcal{E}|_{L\setminus D}$ is injective.
Since $\mathcal{Q}|_{L\setminus E}$ is torsion free, 
we see that $\mathcal{Q}|_{L\setminus E}\to \mathcal{E}|_{L\setminus E}$ is injective.
Hence $(\mathcal{Q}|_L)|_{L\setminus E}\to \im (\varphi|_L)|_{L\setminus E}$ is injective,
and therefore 
$(\mathcal{Q}|_L)|_{L\setminus E}\to \im (\varphi|_L)|_{L\setminus E}$ 
is an isomorphism.

By Theorem~\ref{Usefulottaviani}, 
we see 
that 
$\det\mathcal{S}\cong \mathcal{O}(2^{s-1})$, that $\rk \mathcal{S}=2^s$,
and that $\mathcal{S}$ is $\mu$-stable
with respect to $\mathcal{O}(1)$.
We have $
1=\deg\mathcal{S}/\rk\mathcal{S}
\leqq 
\deg\mathcal{Q}/\rk\mathcal{Q}$,
since the degree $\deg \mathcal{S}$ of $\mathcal{S}$ with respect to $\mathcal{O}(1)$
is $(\det\mathcal{S}).\mathcal{O}(1)^{n-1}=2^s$.

The existence of $\varphi|_{L}\neq 0$ implies 
that $\mathcal{E}|_{L}$ is isomorphic to either $\mathcal{O}(1)\oplus \mathcal{O}^{\oplus r-1}$
or $\mathcal{O}(1,0)\oplus \mathcal{O}(0,1)\oplus \mathcal{O}^{\oplus r-2}$
by Theorem~\ref{Chern1surface}.

Suppose that 
$\mathcal{E}|_{L}$ is isomorphic $\mathcal{O}(1)\oplus \mathcal{O}^{\oplus r-1}$.
Since $\mathcal{S}|_{L}\cong (\mathcal{O}(1,0)\oplus \mathcal{O}(0,1))^{\oplus s}$,
$\im (\varphi|_{L})$ is a 
subsheaf of 
a subsheaf $\mathcal{O}_L(1)$ of $\mathcal{E}|_L$.
Since $\im (\varphi|_L)|_{L\setminus E}$ is isomorphic to $\mathcal{Q}|_{L\setminus E}$,
we see that $\rk \mathcal{Q}=1$.
Let $\mathcal{Q}^{\vee\vee}$ be the reflexive hull of $\mathcal{Q}$.
Then $\mathcal{Q}^{\vee\vee}$ is a line bundle 
and it is a subsheaf of $\mathcal{E}$.
Since 
$
1\leqq \deg\mathcal{Q}=\deg\mathcal{Q}^{\vee\vee}
$,
this implies that $\mathcal{E}$ contains $\mathcal{O}(1)$ as a subsheaf.
This contradicts the assumption that $\Hom(\mathcal{O}(1),\mathcal{E})=0$.
Therefore this case does not occur.

Suppose that 
$\mathcal{E}|_{L}
\cong 
\mathcal{O}(1,0)\oplus \mathcal{O}(0,1)\oplus \mathcal{O}^{\oplus r-2}$.
Since $\mathcal{S}|_{L}\cong (\mathcal{O}(1,0)\oplus \mathcal{O}(0,1))^{\oplus s}$,
$\im (\varphi|_{L})$ is either one of 
$\mathcal{O}(1,0)$, $\mathcal{O}(0,1)$,
or $\mathcal{O}(1,0)\oplus \mathcal{O}(0,1)$.
Since $\det \im (\varphi|_L)|_{L\setminus E}$ is isomorphic to $\det\mathcal{Q}|_{L\setminus E}$,
we see that a morphism  $(\det\mathcal{Q})|_{L}\to \det \im (\varphi|_L)$ of line bundles 
is surjective in codimension two.
Therefore $(\det\mathcal{Q})|_{L}\to \det \im (\varphi|_L)$ is an isomorphism.
Since $\det \im (\varphi|_L)$ is thus the restriction of the line bundle $\det\mathcal{Q}$ on $\mathbb{Q}^n$,
we conclude that $\im(\varphi|_{L})\cong \mathcal{O}(1,0)\oplus \mathcal{O}(0,1)$.
Thus $\mathcal{Q}$ has rank two and $\det\mathcal{Q}\cong\mathcal{O}(1)$.
Hence 
$
\deg\mathcal{Q}/\rk\mathcal{Q}=1
$.
Therefore
we have 
$
\deg\mathcal{S}/\rk\mathcal{S}= \deg\mathcal{Q}/\rk\mathcal{Q}
$,
which implies that $\varphi$ is an isomorphism onto its image $\mathcal{Q}$.
Thus $s=1$, i.e., $n=3$ or $4$.
Since $\mathcal{Q}$ is now a vector bundle, so is $\mathcal{Q}|_{L}$.
Since two vector bundles $\mathcal{Q}|_L$ and $\im (\varphi|_L)$ are isomorphic in codimension one,
we see that $\mathcal{Q}|_L$ and $\im (\varphi|_L)$ are isomorphic.
Since $\im (\varphi|_L)$ is a subbundle of $\mathcal{E}|_L$,
we conclude that $\mathcal{S}$ is a subbundle of $\mathcal{E}$.
Thus $\mathcal{F}$ is a nef vector bundle with $\det\mathcal{F}\cong 0$.
Hence $\mathcal{F}\cong \mathcal{O}^{\oplus r-2}$ and we obtain $\mathcal{E}\cong \mathcal{S}\oplus \mathcal{O}^{\oplus r-2}$.
\end{proof}
Based on our framework, we give a different proof of the following theorem of Peternell-Szurek-Wi\'{s}niewski~\cite[Theorem 2]{pswnef}.

\begin{thm}\label{d=1OnHyperquadric}
Suppose that $d=1$, i.e., that $\det\mathcal{E}\cong \mathcal{O}(1)$. 
Then $\mathcal{E}$ satisfies one of the following:
\begin{enumerate}
\item $\mathcal{E}\cong \mathcal{O}(1)\oplus \mathcal{O}^{\oplus r-1}$.
\item $\mathcal{E}\cong \mathcal{S}\oplus \mathcal{O}^{\oplus r-2}$, where 
$\mathcal{S}$ is a spinor bundle and $n=3$ or $4$.
\item $\mathcal{E}$ fits in an exact sequence
\[
0\to \mathcal{O}(-1)\to \mathcal{O}^{\oplus r+1}\to \mathcal{E}\to 0.
\]
\end{enumerate}
\end{thm}
\begin{proof}
If $d_{\min}<0$, then $r=1$ by Proposition~\ref{firstConstraint} (2) (d).
If $r=1$, then $\mathcal{E}\cong \mathcal{O}(1)$ and $d_{\min}=-1$.
In the following, we assume that $d_{\min}\geqq 0$ and that $r\geqq 2$.
We know that $d_{\min}\leqq 1$ by Corollary~\ref{quadricResol3}.

If $\Hom(\mathcal{O}(1),\mathcal{E})\neq 0$, then $\mathcal{E}\cong \mathcal{O}(1)\oplus \mathcal{O}^{\oplus r-1}$
by Proposition~\ref{dsubExist}. In the following, we assume that $\Hom(\mathcal{O}(1),\mathcal{E})= 0$.

If $\Hom(\mathcal{S},\mathcal{E})\neq 0$, then $n=3$ or $4$,
and $\mathcal{E}\cong \mathcal{S}\oplus \mathcal{O}^{\oplus r-2}$
by Proposition~\ref{dmin0spinor}.
In the following, we assume that $\Hom(\mathcal{S},\mathcal{E})= 0$.

Under the assumptions that $d_{\min}\geqq 0$, that $\Hom(\mathcal{O}(1),\mathcal{E})= 0$, and that 
$\Hom(\mathcal{S},\mathcal{E})= 0$, we have $d_{\min}=1$.
Indeed, if $d_{\min}$ were zero, then 
$\mathcal{E}$ would be isomorphic to $\mathcal{O}^{\oplus r}$ by 
the standard resolution,
which contradicts that $d=1$.

In the following, we assume $d_{\min}=1$.
Since $\Hom(\mathcal{O}(2),\mathcal{E}(1))= 0$, we see that $e_{0,3}=0$ if $n$ is odd,
and that $e_{0,4}=0$ if $n$ is even.
Set $e=\dim H^0(\mathcal{E})$. Then $e_{0,2}=e$ if $n$ is odd,
and $e_{0,3}=e$ if $n$ is even.
By Lemma~\ref{withoutO(1)}, we have $e=\dim H^0(\mathcal{E}|_{\mathbb{Q}^2})$
for any $2$-dimensional smooth linear section $\mathbb{Q}^2$
of $\mathbb{Q}^n$.
Moreover we have $\dim H^0(\mathcal{E}|_{\mathbb{Q}^2})\geqq r+1$ 
by Theorem~\ref{Chern1surface}.
Therefore we see that 
\[e\geqq r+1.\]

Note that $\mathcal{E}$ is globally generated
since $\mathcal{E}|_{\mathbb{Q}^2}$ is globally generated
by Theorem~\ref{Chern1surface} and $H^0(\mathcal{E})\cong H^0(\mathcal{E}|_{\mathbb{Q}^2})$
by Lemma~\ref{withoutO(1)}
for any 
$\mathbb{Q}^2$.
Hence we obtain the desired exact sequence
\[0\to\mathcal{O}(-1)\to \mathcal{O}^{\oplus r+1}\to \mathcal{E}\to 0
\]
if $e=r+1$.

In the following, we shall show that $e=r+1$.
Set $s=\lfloor \frac{n-1}{2}\rfloor$.
We divide the case according to whether $n$ is odd or not.

Suppose that $n$ is odd. Then $n=2s+1$.
The standard resolution of $\mathcal{E}(1)$
modified according to Proposition~\ref{secondconstraint}
is 
\[
0
\to \mathcal{O}^{\oplus e_{2,0}}
\to \mathcal{S}^{\oplus e_{1,1}}\oplus\mathcal{O}^{\oplus (e_{1,0}-e_{0,0})}
\to \mathcal{O}(1)^{\oplus e}\oplus\mathcal{S}^{\oplus e_{0,1}}
\to \mathcal{E}(1)\to 0.
\]
Since $\Hom(\mathcal{S},\mathcal{O}(1))\cong H^0(\mathcal{S})$
and $\dim H^0(\mathcal{S})=2^{s+1}$ by Theorem~\ref{Usefulottaviani}, 
we 
see 
$e_{1,1}=2^{s+1}e$ and 
$e_{2,0}=2^{2s+2}e=2^{n+1}e$.
Since $\det\mathcal{S}\cong \mathcal{O}(2^{s-1})$ by Theorem~\ref{Usefulottaviani}, by looking at $\det(\mathcal{E}(1))$,
we see that 
\[1+r=e+2^{s-1}(e_{0,1}-e_{1,1})=2^{s-1}e_{0,1}+(1-2^{n-1})e.\]
Hence 
\[e_{0,1}=2^{1-s}\{1+r+(2^{n-1}-1)e\}.
\]
Since $\rk \mathcal{S}=2^s$ by Theorem~\ref{Usefulottaviani},
by looking at $\rk\mathcal{E}(1)$, we see that 
\begin{equation*}
\begin{split}
r&=e+2\{1+r+(2^{n-1}-1)e\}-2^ne
-e_{1,0}
+e_{0,0}+2^{n+1}e\\
&=
2r+2+(2^{n+1}-1)e-e_{1,0}+e_{0,0}.
\end{split}
\end{equation*}
Hence $e_{1,0}-e_{0,0}=r+2+(2^{n+1}-1)e$.
Summing up, we have a locally free resolution
\[
0
\to 
\mathcal{P}_2
\to 
\mathcal{P}^{0}_1
\to 
\mathcal{P}^{0}_0
\to \mathcal{E}(1)\to 0,
\]
where 
\begin{equation*}
\begin{split}
\mathcal{P}_2&=\mathcal{O}^{\oplus 2^{n+1}e},\\
\mathcal{P}^{0}_1&=\mathcal{O}^{\oplus r+2+(2^{n+1}-1)e}\oplus\mathcal{S}^{\oplus 2^{s+1}e},\\
\mathcal{P}^{0}_0&=\mathcal{S}^{\oplus 2^{1-s}\{1+r+(2^{n-1}-1)e\}}\oplus\mathcal{O}(1)^{\oplus e}.
\end{split}
\end{equation*}
We split the above long exact sequence into the following two short exact sequences
\begin{equation*}
0
\to \mathcal{P}_2
\to \mathcal{P}^{0}_1
\to \mathcal{G}
\to 0,\quad
0
\to \mathcal{G}
\to \mathcal{P}^{0}_0
\to \mathcal{E}(1)\to 0.
\end{equation*}
Note that $\RHom(\mathcal{S}(1),\mathcal{O}(1))\cong 0$.
It follows from Theorem~\ref{Usefulottaviani} that 
\[\RHom(\mathcal{S}(1),\mathcal{O})\cong \RHom(\mathcal{O}(1),\mathcal{S}^{\vee})
\cong \RHom(\mathcal{O}(2),\mathcal{S}^{\vee}(1))\cong \RHom(\mathcal{O}(2),\mathcal{S})\cong 0.\]
Hence 
$\RHom(\mathcal{S}(1),\mathcal{P}_2)\cong 0$,
 and thus 
\[
\RHom(\mathcal{S}(1),\mathcal{P}^0_1)\cong \RHom(\mathcal{S}(1),\mathcal{G}).
\]
Since $\RHom(\mathcal{S}(1),\mathcal{S})\cong K[-1]$ by Lemma~\ref{S(1)toS},
we see that 
\begin{equation*}
\begin{split}
\RHom(\mathcal{S}(1),\mathcal{P}^{0}_0)
&
\cong K[-1]^{\oplus 2^{1-s}\{1+r+(2^{n-1}-1)e\}},
\\
\RHom(\mathcal{S}(1),\mathcal{P}^{0}_1)
&
\cong K[-1]^{\oplus 2^{s+1}e}.
\end{split}
\end{equation*}
Therefore we obtain a distinguished triangle
\[
K[-1]^{\oplus 2^{s+1}e}
\to
K[-1]^{\oplus 2^{1-s}\{1+r+(2^{n-1}-1)e\}}
\to
\RHom(\mathcal{S}(1),\mathcal{E}(1))
\to.
\]
Since we assume now that $\Hom (\mathcal{S},\mathcal{E})=0$,
we get an exact sequence
\[
0
\to 
K^{\oplus 2^{s+1}e}
\to
K^{\oplus 2^{1-s}\{1+r+(2^{n-1}-1)e\}}
\to
\Ext^1(\mathcal{S}(1),\mathcal{E}(1))
\to
0.
\]
In particular we have $2^{s+1}e\leqq 2^{1-s}\{1+r+(2^{n-1}-1)e\}$, i.e., $e\leqq r+1$.
Hence 
\[e=r+1.
\]

Suppose that $n$ is even. Then $n=2s+2$.
The standard resolution of $\mathcal{E}(1)$
modified according to Proposition~\ref{secondconstraint}
is 
\[
0
\to \mathcal{P}_2
\to \mathcal{P}^{0}_1
\to \mathcal{P}^{0}_0
\to \mathcal{E}(1)
\to 0,
\]
where 
\begin{equation*}
\begin{split}
\mathcal{P}_2&=\mathcal{O}^{\oplus e_{2,0}},\\
\mathcal{P}^{0}_1&=\mathcal{O}^{\oplus e_{1,0}-e_{0,0}}\oplus (\mathcal{S}^+)^{\oplus e_{1,1}}\oplus(\mathcal{S}^-)^{\oplus e_{1,2}},\\
\mathcal{P}^{0}_0&=(\mathcal{S}^+)^{\oplus e_{0,1}}\oplus(\mathcal{S}^-)^{\oplus e_{0,2}}\oplus \mathcal{O}(1)^{\oplus e}.\\
\end{split}
\end{equation*}
By Theorem~\ref{Usefulottaviani},
we see that 
$\Hom(\mathcal{S^+},\mathcal{O}(1))\cong H^0(\mathcal{S^-})$ 
and $\Hom(\mathcal{S^-},\mathcal{O}(1))\cong H^0(\mathcal{S^+})$ 
if $s$ is 
even,
and that $\Hom(\mathcal{S^+},\mathcal{O}(1))\cong H^0(\mathcal{S^+})$
and $\Hom(\mathcal{S^-},\mathcal{O}(1))\cong H^0(\mathcal{S^-})$ if $s$ is 
odd.
In the following, we denote $\mathcal{S}^+$ and  $\mathcal{S}^-$ simply by $\mathcal{S}$ 
if no confusion occurs.
Since $\dim H^0(\mathcal{S})=2^{s+1}$ by Theorem~\ref{Usefulottaviani}, 
we see that $e_{1,1}=2^{s+1}e$,
that $e_{1,2}=2^{s+1}e$, and that $e_{2,0}=2^{2s+3}e=2^{n+1}e$.
Since $\det\mathcal{S}\cong \mathcal{O}(2^{s-1})$ by Theorem~\ref{Usefulottaviani}, by looking at $\det(\mathcal{E}(1))$,
we see that 
\[1+r=e+2^{s-1}(e_{0,1}+e_{0,2}-e_{1,1}-e_{1,2})=2^{s-1}(e_{0,1}+e_{0,2})+(1-2^{n-1})e.\]
Hence 
\[e_{0,1}+e_{0,2}=2^{1-s}\{1+r+(2^{n-1}-1)e\}.
\]
Since $\rk \mathcal{S}=2^s$ by Theorem~\ref{Usefulottaviani},
by looking at $\rk\mathcal{E}(1)$, we see that 
\begin{equation*}
\begin{split}
r&=e+2\{1+r+(2^{n-1}-1)e\}-2^{n}e
-e_{1,0}
+e_{0,0}+2^{n+1}e\\
&=
2r+2+(2^{n+1}-1)e-e_{1,0}+e_{0,0}.
\end{split}
\end{equation*}
Hence $e_{1,0}-e_{0,0}=r+2+(2^{n+1}-1)e$.
Summing up, we have a locally free resolution
\[
0
\to 
\mathcal{P}_2
\to 
\mathcal{P}^{0}_1
\to 
\mathcal{P}^{0}_0
\to \mathcal{E}(1)\to 0,
\]
where 
\begin{equation*}
\begin{split}
\mathcal{P}_2&=\mathcal{O}^{\oplus 2^{n+1}e},\\
\mathcal{P}^{0}_1&=\mathcal{O}^{\oplus r+2+(2^{n+1}-1)e}\oplus(\mathcal{S}^+)^{\oplus 2^{s+1}e}\oplus(\mathcal{S}^-)^{\oplus 2^{s+1}e},\\
\mathcal{P}^{0}_0&=(\mathcal{S}^+)^{\oplus e_{0,1}}\oplus(\mathcal{S}^-)^{\oplus e_{0,2}}\oplus \mathcal{O}(1)^{\oplus e},
\textrm{ \quad }e_{0,1}+e_{0,2}=2^{1-s}\{1+r+(2^{n-1}-1)e\}.
\end{split}
\end{equation*}
We split the above long exact sequence into the following two short exact sequences
\begin{equation*}
0
\to \mathcal{P}_2
\to \mathcal{P}^{0}_1
\to \mathcal{G}
\to 0, \quad
0
\to \mathcal{G}
\to \mathcal{P}^{0}_0
\to \mathcal{E}(1)\to 0.
\end{equation*}
Note that $\RHom(\mathcal{S}(1),\mathcal{O}(1))\cong 0$
and that 
$\RHom(\mathcal{S}(1),\mathcal{O})\cong 0$ as in the odd-dimensional case.
Hence $\RHom(\mathcal{S}(1),\mathcal{P}_2)\cong 0$, and thus 
$\RHom(\mathcal{S}(1),\mathcal{P}^0_1)\cong \RHom(\mathcal{S}(1),\mathcal{G})$.
Since both 
$\RHom((\mathcal{S}^+\oplus \mathcal{S}^-)(1),\mathcal{S}^+)$
and 
$\RHom((\mathcal{S}^+\oplus \mathcal{S}^-)(1),\mathcal{S}^-)$
are isomorphic to $K[-1]$
by Lemma~\ref{S(1)toS},
we see that 
\begin{equation*}
\begin{split}
\RHom((\mathcal{S}^+\oplus \mathcal{S}^-)(1),\mathcal{P}^{0}_0)
&
\cong K[-1]^{\oplus 2^{1-s}\{1+r+(2^{n-1}-1)e\}},
\\
\RHom((\mathcal{S}^+\oplus \mathcal{S}^-)(1),\mathcal{P}^{0}_1)
&
\cong K[-1]^{\oplus 2^{s+2}e}.
\end{split}
\end{equation*}
Therefore we obtain a distinguished triangle
\[
K[-1]^{\oplus 2^{s+2}e}
\to
K[-1]^{\oplus 2^{1-s}\{1+r+(2^{n-1}-1)e\}}
\to
\RHom((\mathcal{S}^+\oplus \mathcal{S}^-)(1),\mathcal{E}(1))
\to.
\]
Since we assume now that $\Hom (\mathcal{S}^+\oplus \mathcal{S}^-,\mathcal{E})=0$,
we get an exact sequence
\[
0
\to 
K^{\oplus 2^{s+2}e}
\to
K^{\oplus 2^{1-s}\{1+r+(2^{n-1}-1)e\}}
\to
\Ext^1((\mathcal{S}^+\oplus \mathcal{S}^-)(1),\mathcal{E}(1))
\to
0.
\]
In particular we have $2^{s+2}e\leqq 2^{1-s}\{1+r+(2^{n-1}-1)e\}$, i.e., $e\leqq r+1$.
Hence 
\[e=r+1.
\]
\end{proof}

\bibliographystyle{plain}
\bibliography{NefOnHyperquadrics.bbl}
\end{document}